\documentclass[12pt,reqno]{amsart}

\usepackage{amssymb}
\usepackage{amsthm}
\usepackage{amsmath}
\usepackage{amsaddr}
\usepackage{bm}
\usepackage{cite}
\usepackage{mathrsfs}
\usepackage{xcolor}

\usepackage[OT1]{fontenc}
\RequirePackage{times}

\usepackage[left=2.3cm, right=2.3cm, top=3cm]{geometry}

\allowdisplaybreaks
\newtheorem{theorem}{Theorem}[section]

\newtheorem{lemma}[theorem]{Lemma}

\theoremstyle{remark}
\newtheorem{remark}[theorem]{Remark}

\def\vp{\varphi}

\def\n{\textbf{\textit{n}}}
\def\R3{\mathbb{R}^3}
\def\F2o{\overline{F_2}}
\def \N{\mathcal{N}}
\def \An{A^{-1}}
\def \E{E}

\def\d{{\rm d}}
\def \l {\langle}
\def \r {\rangle}
\def\V{{\boldsymbol{V}}}

\def\H{\mathbf{H}}

\def\u{\textbf{\textit{u}}}
\def\uu{\textbf{\textit{u}}}
\def\vv{\textbf{\textit{v}}}
\def\ww{\textbf{\textit{w}}}
\def\ddt{\frac{\d}{\d t}}

\def\vp{\varphi}

\def\P{\mathbb{P}}

\def \au {\rm}
\def \ti {\it}
\def \jou {\rm}
\def \bk {\it}
\def \no#1#2#3 {{\bf #1} (#3), #2.}
\def \eds#1#2#3 {#1, #2, #3.}
\def \nome#1#2 {{\bf #1}, (#2).}

\begin{document}

\title[Analysis of the Hele-Shaw-Cahn-Hilliard system]
{Well-posedness of a diffuse interface model \\ for Hele-Shaw flows}

\author[Andrea Giorgini]{Andrea Giorgini}
\address{Department of Mathematics, Indiana University\\
Bloomington, IN 47405, USA\\
agiorgin@iu.edu}





\begin{abstract}
\noindent
We study a diffuse interface model describing the motion of two viscous fluids driven by the surface tension in a Hele-Shaw cell. The full system consists of the Cahn-Hilliard equation coupled with the Darcy's law. We address the physically relevant case in which the two fluids have different viscosities (unmatched viscosities case) and the free energy density is the logarithmic Helmholtz potential. In dimension two we prove the uniqueness of weak solutions under a regularity criterion, and the existence and uniqueness of global strong solutions. In dimension three we show the existence and uniqueness of strong solutions, which are local in time for large data or global in time for appropriate small data. These results extend the analysis obtained in the matched viscosities case by {\it Giorgini, Grasselli \& Wu (Ann. Inst. H. Poincar\'{e} Anal. Non Lin\'{e}aire \textbf{35} (2018), 318-360)}. 
Furthermore, we prove the uniqueness of weak solutions in dimension two by taking the well-known polynomial approximation of the logarithmic potential.
\end{abstract}

\maketitle


\section{Introduction}

In this paper we consider a diffuse interface model for binary fluids flows driven by the surface tension between two flat plates separated by a narrow gap. This is known in literature as Hele-Shaw-Cahn-Hilliard (HSCH) system \cite{DGL,LLG2002,LLG2002-2}. 
In a bounded domain $\Omega\subset \mathbb{R}^d$, $d=2,3$, with smooth boundary $\partial \Omega$, the dynamics of the difference of the fluids concentrations $\vp$ is governed by the Cahn-Hilliard equation that reads as
\begin{equation}
\label{CH}
\partial_t \vp+ \u \cdot \nabla \vp =   \Delta (- \Delta \vp +  \Psi'(\vp)),
\end{equation}
where the fluid velocity $\uu$ and the pressure $p$ are given by the Darcy's law 
\begin{equation}
\label{DL}
\nu(\vp) \u  +\nabla p = (-\Delta \vp+\Psi'(\vp)) \nabla \vp, \quad 
\mathrm{div}\, \uu=0.
\end{equation}
Here, $\nu$ is the viscosity of the mixture and $\Psi$ is the free energy density of mixing. Some physical parameters have been scaled to one for simplicity. By virtue of its definition, the state variable $\vp$ takes values in the interval $[-1,1]$, where $1$ and $-1$ are the homogeneous states (pure concentrations). The term $\Delta(-\Delta \vp+\Psi'(\vp))$ in \eqref{CH} accounts for diffusion mechanisms due to mixing in the system. The Korteweg term $(-\Delta \vp+\Psi'(\vp))\nabla \vp$ in \eqref{DL}, which can be rewritten as $-\mathrm{div}\, (\nabla \vp\otimes \nabla \vp)$ (up to redefine the pressure), models capillary forces due to the surface tension. 
Introducing the chemical potential $\mu=-\Delta \vp+\Psi'(\vp)$,
the HSCH system is rewritten as
\begin{equation}
\begin{cases}
\label{CHHS}
\nu(\vp) \u  +\nabla p = \mu \nabla \vp,\\
\mathrm{div}\, \u=0, \\
\partial_t \vp+ \u \cdot \nabla \vp =   \Delta \mu,\\
\mu= - \Delta \vp +  \Psi'(\vp),
\end{cases}
 \quad \text{ in } \Omega\times (0,T).
\end{equation}
The system is associated with the impermeability condition for $\uu$ and homogeneous Neumann boundary conditions for both $\vp$ and $\mu$, and an initial condition. This corresponds to the conditions
\begin{equation}
\label{bdini}
\begin{cases}
\u \cdot \n= \partial_\n \mu=\partial_\n \vp=0, \quad &\text{ on } \partial \Omega \times (0,T),\\
\vp( \cdot,0)=\vp_0, \quad &\text{ in } \Omega,
\end{cases}
\end{equation}
where $\n$ is the unit outward normal vector to the boundary $\partial \Omega$. 

The purpose of this contribution is to study the HSCH system \eqref{CHHS}-\eqref{bdini} under physically grounded assumptions on the viscosity of the mixture $\nu$ and on the free energy density $\Psi$. In the theory of mixtures the viscosity coefficient $\nu$ is an expression of the concentration $\vp$, the viscosities $\nu_1$ and $\nu_2$ of the homogeneous fluids and other parameters, such as the temperature. Particular relations are usually validated through experiments. A typical approximation in the unmatched viscosities case ($\nu_1\neq\nu_2$) can be expressed by the linear combination
\begin{equation}
\label{vis}
\nu(s)= \nu_1 \frac{1+s}{2}+\nu_2 \frac{1-s}{2}, \quad  s \in [-1,1].
\end{equation}   
Throughout our analysis, we will assume that $\nu$ is a smooth and strictly positive function (see (A1)).
From thermodynamics theory the Helmholtz free energy density $\Psi$ is given by $\Psi=\Delta U-\theta \Delta S$, where $\Delta U$ and $\Delta S$ are the variations of internal energy and entropy of the mixture after mixing, and $\theta$ is the constant temperature of the system. In the case of regular or polymer solutions, the entropy of mixing is derived from the Boltzmann equation, while the internal energy depends on the configuration after mixing (see, e.g., \cite{LHJSYK2014}). This leads to the well-known Helmholtz potential   
\begin{equation}
\label{SING}
\Psi(s)=\frac{\theta}{2}\left[ (1+s)\log(1+s)+(1-s)\log(1-s)\right]-\frac{\theta_0}{2} s^2,
 \quad s \in [-1,1],
\end{equation}
where $\theta_0$ is the so-called critical temperature. We consider hereafter the interesting case in which $0<\theta<\theta_0$.
The total free energy associated with system \eqref{CHHS} is the Ginzburg-Landau functional
\begin{equation}
\label{energy}
\E(\vp)= \int_\Omega  \frac12 |\nabla \vp|^2
+ \Psi(\vp) \, \d x,
\end{equation}
which accounts for interfacial energy and mixing tendencies of the binary mixture. We observe that any sufficiently regular solution satisfies the energy balance
\begin{equation}
\label{energybalance}
 \E(\vp(t))+ \int_0^t \int_{\Omega} |\nabla \mu|^2 + \nu(\vp)|\uu|^2 \, \d x \, \d \tau =\E(\vp_0).
\end{equation}

\subsection{Physical Background}
A fundamental problem in fluid mechanics concerns the dynamics of two adjoining fluids. 
Complex phenomena already appear in simple experiments when the spatial regions occupied by a single flow is deformed by moving fluid structures, and the interface area decreases its characteristic length scale and even changes its topology. An important example is the motion of one or two fluids between flat plates separated by a narrow gap. This was proposed by Henry S. Hele-Shaw in the seminal experiment \cite{HELESHAW1898} aiming to describe fluid flows in which viscous forces prevail over inertial forces. Under this assumption, the Navier-Stokes equations reduce to a linear relation, the so-called Darcy's law, that reads as 
\begin{equation}
\label{darcy}
\overline{\uu}=- \frac{h^2}{12 \nu} \nabla p.
\end{equation}
Here $\overline{\uu}$ is a two dimensional vector field denoting the average of the velocity over the cell gap $h$, namely
$\overline{u}_i= \frac{1}{h} \int_0^h u_i \, \d x_3$, for $i=1,2$.
Even though the derivation of \eqref{darcy} from the Navier-Stokes equations naturally leads to a two dimensional flow (see, e.g. \cite{GV2006}), the equations \eqref{darcy} have the same form of the Darcy's law studied for saturated flow in porous media in three dimensions.

The theory of diffuse interface for fluid mixtures represents nowadays a successful method to simulate complex systems, being able to capture the main features of the mutual interplay in the motion of two fluids. The key concept is to represent the interfaces as regions with finite thickness, which are described as the level-set of the difference of fluids concentrations $\vp$, whose values are uniform in homogeneous states and have a rapid but smooth variation across the interface. The dynamics of $\vp$ is derived from the mass balance of the mixture assuming a partial mixing at the interface. This leads to the Cahn-Hilliard equation \eqref{CH}, in which the diffusive mass flux is given in term of the derivative of the free energy \eqref{energy}.
The full model consisting of additional equations for the velocity field is derived through an energetic variational procedure. In contrast to the sharp interface method, in which the interface is a time-dependent surface, the main advantage of the diffuse interface formulation is the transformation from a Lagrangian to an Eulerian description, which allows large deformation and topological changes of the interfaces. For this reason, diffuse interface models have been employed in many applications, as witnessed by a vast literature mostly devoted to numerical simulations. When the interface thickness go to zero, the relating free boundary (or sharp interface) problems are formally recovered from the diffuse interface system. We refer the interested reader to the reviews \cite{AMW,Review2005} and the references therein. 

In the diffuse interface theory, the Hele-Shaw-Cahn-Hilliard model \eqref{CHHS}-\eqref{bdini} has been derived as a simplification of the Navier-Stokes-Cahn-Hilliard system (Model H) in \cite{LLG2002,LLG2002-2} and, more recently, in \cite{DGL}. 
In these papers, this model has been applied to investigate pinchoff and reconnection, rising bubbles and fingering instabilities in a Hele-Shaw cell. The HSCH system has been generalized in \cite{HSW2014} for flows in karstic geometric. In these last years, the HSCH model has also had a considerable impact in modeling tumor growths. This system has been coupled with reaction-diffusion equations to take chemotaxis, active transport and nutrients into account. Among the large literature devoted to this subject, we mention \cite{CWSL2014,CLLW2009,LFJCLMWC2010,GLNS2017,GLSS2016}. 

\subsection{Summary of Previous Results}
The mathematical analysis of Hele-Shaw-Cahn-Hilliard system with constant viscosity and logarithmic free energy density has been addressed in \cite{GGW2018}. The authors proved the existence of global weak solutions in dimension two and three. The uniqueness of weak solutions, their global regularity and the instantaneous separation property (from pure concentrations) have been established  in dimension two. Furthermore, the existence of global regular solutions is shown in dimension three providing that the initial condition is close to an energy minimizer. We also mention two results concerning the existence of global weak solutions proved in \cite{DFRSS} and \cite{FLRS2018}. In the former the HSCH system is coupled with a transport equation and a quasi-static reaction-diffusion equation. In the latter a multi-species HSCH model is studied with a mass source $S$ depending linearly on the concentration vector $(\vp_c,\vp_d)$.
In the recent paper \cite{DPGG2018}, the authors considered the nonlocal version of the HSCH system with constant viscosity and logarithmic potential. The global existence and uniqueness of weak and strong solutions are proven in both dimensions two and three. 
Moreover, the instantaneous separation property is established in dimension two. It is worth noting that, in comparison with the literature for the local case,  the analysis in \cite{DPGG2018} provides a first result of global well-posedness in dimension three.

The HSCH system has been studied considering the double-well polynomial approximation $\Psi_0(s)= (s^2-1)^2$ of the logarithmic Helmholtz potential (up to a constant) in previous works.
For periodic boundary conditions, in the unmatched viscosities case  existence and uniqueness of strong solutions in $H^s(\Omega)$, with $s>1+ \frac{d}{2}$, global if $d=2$ and local if $d=3$, have been proved in \cite{WZ13}. The convergence to equilibrium and the global existence for initial data close to energy minimizer if $d=3$ were established in \cite{WW12}. The existence of global weak solutions has been shown in \cite{DGL}. In the matched viscosity case ($\nu(s)\equiv 1$), the existence and uniqueness of strong solutions in $H^2(\Omega)$, global if $d=2$ and local if $d=3$, have been established in \cite{LTZ}. The authors also studied Gevrey regularity and 
exponential stability of a constant state (i.e., the average of total mass for the initial datum) under suitable smallness assumptions. In \cite{Fei2016} and \cite{MR2018}, it is shown the convergence of weak solutions to \eqref{CHHS} to varifold solutions of the associated sharp interface problem.
We mention that some variants of the HSCH system has been investigated in \cite{HWW} and \cite{JWZ} (see also \cite{SW2018} for a related optimal control problem).
As already pointed out in \cite{GGW2018} (see also \cite{SW2018}), the uniqueness of weak solutions has not been proven yet in dimension two. Finally, we remark that the main drawback in the analysis with a polynomial-like potential $\Psi_0$ is the lack of physical solutions. More precisely, it is not possible to  guarantee that $\vp$ stays in the meaningful interval $[-1,1]$ for any initial condition.  


\subsection{Main Results} 
The aim of this contribution is to present a mathematical theory of existence, uniqueness and regularity for the Hele-Shaw-Cahn-Hilliard system with unmatched viscosities and thermodynamically consistent logarithmic free energy density. After discussing the existence of global weak solutions, we prove the following results:
\begin{itemize}
\item Uniqueness criterion for weak solutions in dimension two (Theorem \ref{weak-strong});
\item Existence and uniqueness of global strong solutions in dimension two (Theorem \ref{strong2d});
\item Existence and uniqueness of strong solutions in dimension three, local in time for large data or global in time for small data (Theorem \ref{strong3d}). 
\end{itemize}
We point out that the first two results implies the weak-strong uniqueness property in dimension two, which allows us to improve the global regularity of any weak solution. Furthermore, we give a positive answer to a question related to the HSCH system with regular (polynomial) potential, namely we show
\begin{itemize}
\item Uniqueness of weak solutions in dimension two (Theorem \ref{RU}).
\end{itemize}

The main issue in the analysis of the HSCH system is not merely the strong coupling between the Darcy's law and the Cahn-Hilliard equation given by the capillary forces term $-\mathrm{div}\, (\nabla \vp\otimes \nabla \vp)$, whose regularity determines roughly the properties of $\uu$. Rather the crucial difficulty arises from the combination of this nonlinear term with the non-constant viscosity and the singular nature of the logarithmic potential \eqref{SING}. We note that the latter makes challenging the control of the derivatives of the convex part of the potential $F(s)=\frac{\theta}{2}\left[ (1+s)\log(1+s)+(1-s)\log(1-s)\right]$. Indeed, the different growth of its derivatives when $s$ approach $\pm 1$ translates into the relations
\begin{equation}
\label{growth}
F''(s)\leq \mathrm{e}^{C|F'(s)|} \quad \text{and} \quad |F'''(s)|\leq CF''(s)^2, \quad \forall \, s \in (-1,1),
\end{equation}
which prevent a control $\Psi''(\vp)$ in terms of $L^p$-norm of $\Psi'(\vp)$. 
As a consequence, the difference of solutions $\vp_1-\vp_2$ can only be estimated in the dual space of $H^1(\Omega)$ (cf. \eqref{Idef2} below). In turn, a control of the difference of velocities in $L^2(\Omega)$ is not sufficient due to the term $-\Delta \vp \nabla \vp$ in \eqref{DL}. Another remarkable difficulty due to \eqref{growth} concerns the regularity of the solution. More precisely, the spatial regularity of $\vp$ is at most in $W^{2,p}(\Omega)$, where $p$ depends on the spatial dimension. Thus, further regularity properties compared to $L^2(\Omega \times (0,T))$ for the velocity are a hard task. These issues are overcome in our analysis developing two novel techniques. First, we reformulate the convective term in \eqref{CH} by exploiting the algebraic form of the Darcy's law.   
This method is used in \cite{GGW2018} by noting that $\uu= \P (-\mathrm{div} (\nabla \vp\otimes \nabla \vp))$, where $\P$ is the Leray projection. However, in contrast to the case with matched viscosities studied in \cite{GGW2018}, the presence of the non-constant viscosity gives rise to more complicated terms including the modified pressure $p^\ast$ (see \eqref{I2} below). The key idea in order to avoid a loss of derivative is to rewrite the terms $Z_i$ in \eqref{I2} in such a way that $\nabla \vp$ is the highest order derivative of the solution. This is carried out by exploiting the homogeneous Neumann boundary conditions to cancel all the boundary terms arising from integration by parts. This argument allows us to prove the uniqueness of weak solutions satisfying $\vp\in L^\infty(0,T;W^{1,r}(\Omega))$ for some $r>2$ in dimension two and the uniqueness of strong solutions in dimension three. Second, we show new 	{\it a priori} estimates in order to prove the existence of regular solutions. These are based on a differential equality involving the $L^2$-norm of $\nabla \mu$ and $\uu$, which is combined with elliptic estimates for the Neumann problem with logarithmic nonlinearity and a bound of the vorticity $\mathrm{curl}\, \uu$ for the Darcy's law. In two dimensions, taking advantage of the Br\'{e}zis-Gallouet-Wainger inequality, we demonstrate a logarithmic differential inequality which implies global bounds in time. In three dimensions, the order of the nonlinear terms is supercritical. Since super-quadratic terms arise on the right-hand side of the resulting differential inequality, we infer the local existence of strong solutions for large smooth data.
Nevertheless, exploiting the dissipative mechanims of the system, these super-quadratic terms can be controlled providing that the initial data is suitably small. This entails the existence of global strong solutions and their exponential decay in time for such small data. We point out that our argument simplifies the proof in \cite{GGW2018} based on the \L ojasiewicz-Simon inequality in dimension three.

\medskip

\textbf{Plan of the paper.} In Section \ref{S2} we collect some preliminary results. In Section \ref{S3} we recall the main assumptions and we discuss the existence of weak solutions. Section 
\ref{S4} is devoted to uniqueness results of weak solutions in dimension two. In Section \ref{S5} we study existence and uniqueness of strong solutions and further regularity properties in dimension two. Section \ref{S6} is devoted to the analysis of strong solutions in dimension three. In Section \ref{S7} we provide some remarks and future directions. In Appendix \ref{App} we report some generalized Gronwall lemmas.


\section{Mathematical Setting}
\label{S2}
\setcounter{equation}{0}

\subsection{Function spaces}
Let $X$ be a (real) Banach or Hilbert space, whose norm is denoted
by $\|\cdot\|_X$. The space $X'$ indicates the
dual space of $X$ and $\l \cdot,\cdot\r$ denotes the duality product. The vectorial space $X^d$ endowed with the product structure ($d$ is the spatial dimension) is denoted by $\mathbf{X}$ with norm $\|\cdot \|_{X}$. 
In a bounded domain $\Omega \subset \mathbb{R}^d$ with smooth boundary $\partial \Omega$, $W^{k,p}(\Omega)$, $k\in \mathbb{N}$ and
$p\in [1,+\infty]$, are the Sobolev spaces of real measurable functions on $\Omega$.  We
denote by $H^k(\Omega)$ the Hilbert spaces
$W^{k,2}(\Omega)$ and by $\|\cdot\|_{H^k(\Omega)}$ its norm.
In particular, $H=L^2(\Omega)$ with inner product and norm denoted by $(\cdot,\cdot)$ and
$\|\cdot\|$, respectively. The space $V=H^{1}(\Omega)$ is endowed with the norm $\|f\|_V^2=\|\nabla f \|^2+\|f\|^2$.
For every $f\in V'=(H^1(\Omega))'$, we denote by $\overline{f}$ the total mass of $f$ defined by $\overline{f}=\frac{1}{|\Omega|}\l f,1\r$.
We recall the following Poincar\'{e}'s inequality
\begin{equation}
\label{poincare}
\|f-\overline{f}\|\leq C\|\nabla f\|,\quad \forall\,
f\in V,
\end{equation}
where the constant $C$ depends only on $d$ and $\Omega$.
Next, we introduce the Hilbert space of soleinodal function
$\mathbf{H_\sigma}=\lbrace {\uu\in \mathbf{L}^2(\Omega):
\mathrm{div}\, \uu=0 \text{ in }\Omega, \, \u\cdot
\n=0 \text{ on }\ \partial \Omega\rbrace},
$
endowed with the usual norm $\|\cdot\|$.
Let $\mathbb{P}$ be the Helmholtz-Leray orthogonal projection from
$\mathbf{H}$ onto $\mathbf{H_\sigma}$.
It is well known that every vector field $\uu \in \mathbf{H}$
can be uniquely represented as
$
\u= \vv + \nabla p,
$
where $\vv= \P \u \in \mathbf{H_\sigma}$ and $p\in V$ such that $\overline{p}=0$.  
We recall that $\P$ is a bounded
operator from $\mathbf{W}^{k,p}(\Omega)$,  for $1< p<\infty$ and $k\geq 0$, into itself (cf. \cite[Lemma 3.3]{GM}), namely there exists a constant $C>0$ such that
\begin{equation}
\label{O}
\| \P \u\|_{W^{k,p}(\Omega)}\leq C \| \u\|_{W^{k,p}(\Omega)},
\quad \forall\, \u\in \mathbf{W}^{k,p}(\Omega).
\end{equation} 
In addition, there exists a constant $C>0$ such that
\begin{equation}
\label{rot}
\| \u\|_{V}\leq
C \left( \| \mathrm{curl}\, \uu \| +\| \u\|\right),
\quad \forall \,  \u \in \mathbf{V}\cap \mathbf{H}_\sigma,
\end{equation}
where $\mathrm{curl}\, \uu$  is the vorticity of $\uu$ defined by
$$
\mathrm{curl}\, \uu= \frac{\partial u_2}{\partial x_1}-\frac{\partial u_1}{\partial x_2}\quad d=2, \quad \mathrm{curl}\, \uu=\Big( \frac{\partial u_3}{\partial x_2}- \frac{\partial u_2}{\partial x_3},\frac{\partial u_1}{\partial x_3}- \frac{\partial u_3}{\partial x_1}, \frac{\partial u_2}{\partial x_1}- \frac{\partial u_1}{\partial x_2}\Big)  \quad d=3.
$$

\smallskip

\subsection{Interpolation and product inequalities.}
\label{interpolations} 
We recall here some well-known interpolation inequalities
in Sobolev spaces which can be found in classical
literature (see e.g. \cite{Brezis-Wainger,Te01}):
\begin{itemize}
\item[$\diamond$] Ladyzhenskaya's  inequalities
\begin{align}
\label{L}
\|f\|_{L^4(\Omega)}
\leq C  \| f\|^{\frac12} \|f\|_{V}^{\frac12}, \quad &&\forall
\, f \in V, \  d=2,\\
\label{L3}
\|f\|_{L^4(\Omega)}
\leq C  \| f\|^{\frac14} \|f\|_{V}^{\frac34}, \quad &&\forall
\, f \in V, \  d=3.
\end{align}
\item[$\diamond$] Agmon's inequalities
\begin{align}
&\|f\|_{L^\infty(\Omega)}
\leq C\|f\|^\frac12\|f\|_{H^2(\Omega)}^\frac12,
\quad &&\forall \, f \in H^2(\Omega), \  d=2,\label{Ad2}\\
&\|f\|_{L^\infty(\Omega)}
 \leq C\|f\|_{V}^\frac12\|f\|_{H^2(\Omega)}^\frac12,
 \quad &&\forall \, f \in H^2(\Omega), \  d=3. \label{Ad3}
\end{align}
\item[$\diamond$] Br\'{e}zis-Gallouet-Wainger inequality
\begin{equation} 
\label{BWd2}
\| f\|_{L^\infty(\Omega)}\leq C 
\| f\|_V 
\log^{\frac12} (e +\|f\|_{W^{1,q}(\Omega)}), \quad \forall \, f \in W^{1,q}(\Omega), \ q>2, \  d=2. 
\end{equation}

\item[$\diamond$] Gagliardo-Nirenberg inequalities
\begin{align}
\label{GN2}
&\| f\|_{L^p(\Omega)}\leq C \| f\|_{L^q(\Omega)}^{1-\theta} \| f\|_{V}^\theta, \quad &&\forall \, f \in V, 1\leq q\leq p<\infty, \ \theta= 1-\frac{q}{p}, \ d=2,\\
\label{GN3}
&\| f\|_{L^\infty(\Omega)}\leq C \| f\|^{1-\theta} \| f\|_{W^{1,q}(\Omega)}^\theta, \quad &&\forall \, f \in W^{1,q}(\Omega), \ q>3, \ \theta= \frac{3q}{5q-6}, \ d=3.
\end{align}

\end{itemize}
We report the following results on the differentiation of a product in
Sobolev spaces ($d=2,3$)
\begin{align}
\label{prodV}
&\| fg\|_V\leq C \big( \|f\|_{V}\| g\|_{L^{\infty}(\Omega)}+ \| f\|_{L^{\infty}(\Omega)}
\| g\|_V \big), \quad &&\forall \, f, g \in V\cap L^\infty(\Omega),\\
\label{prodH2}
&\| fg \|_{H^2(\Omega)} \leq C \big( \|f\|_{H^2(\Omega)}\| g\|_{L^{\infty}(\Omega)}+ \| f\|_{L^{\infty}(\Omega)}
\| g\|_{H^2(\Omega)} \big), \quad &&\forall \, f, g \in H^2(\Omega).
\end{align}
\smallskip

\subsection{Neumann's problems for Laplace operator}
We report here some existence results and elliptic estimates regarding homogeneous Neumann problems with constant and non-constant coefficients.
\smallskip

\noindent
\textbf{Case I: Constant coefficients.}
Let us introduce the linear spaces
$$
V_0=\{ u \in V:\ \overline{u}=0\},\quad
L^2_0=\{ u \in H:\ \overline{u}=0\}, \quad
V_0'= \{ u \in V':\ \overline{u}=0 \}.
$$
We consider the linear operator
$A\in \mathcal{L}(V_0,V_0')$ defined by
$
\l A u,v \r= (\nabla u , \nabla v)$,
for all $u,v  \in V_0$.
The operator $A$ is positive, self-adjoint and has compact inverse denoted by $\An$. For $f\in V_0'$, $u=\An f$ is the unique
weak solution of the homogeneous Neumann problem 
$$
\begin{cases}
-\Delta u=f, \quad \text{in} \ \Omega,\\
\partial_\n u=0, \quad \ \  \text{on}\ \partial \Omega,
\end{cases}
$$
namely $\l A u, v \r= \l f,v \r$, for all $v\in V_0$.  
It follows that
\begin{equation}
\l  f, \An g\r =\l \An f, g\r = \int_{\Omega} \nabla(\An f)
\cdot \nabla (\An g) \, \d x, \quad \forall \, f,g \in V_0'.\label{propN2}
\end{equation}
For any $f\in V_0'$, we define $\|f\|_{V_0'}=\|\nabla \An f\|$, which is a norm on $V_0'$ 
equivalent to the natural norm. 
Moreover, the operator $A$ can be seen as an unbounded operator on $L_0^2$ with domain $D(A)= \lbrace u\in V_0\cap H^2(\Omega): \partial_\n u=0 \text{ on } \partial \Omega \rbrace$.
Finally, we report the following Hilbert interpolation inequality and elliptic estimates for the Neumann problem:
\begin{align}
&\|f\| \leq \|f\|_{V_0'}^{\frac12} \| \nabla f\|^{\frac12},
 &&\forall\, f \in V_0,\label{I}\\
&\|\nabla \An f\|_{H^{k}(\Omega)} \leq C \|f\|_{H^{k-1}(\Omega)},
 &&\forall\, f\in H^{k-1}(\Omega)\cap L^2_0,\quad k\in \mathbb{N},\label{N}\\
&\| \An f\|_{W^{k+2,p}(\Omega)}\leq C \| f \|_{W^{k,p}(\Omega)}, 
&&\forall \, f \in W^{k,p}(\Omega)\cap L^2_0,\quad k\in \mathbb{N}, \ 1<p<\infty.
\label{NR}
\end{align}
\smallskip

\noindent
\textbf{Case II: Non-constant coefficients.}
We consider the homogeneous Neumann problem with a non-constant coefficient $K$ depending on a given measurable function $\theta$. This reads as follows
\begin{equation}
\label{NP-2}
\begin{cases}
- \mathrm{div}\, ( K(\theta) \nabla u) =f,\quad &\text{ in }\Omega,\\
\partial_\n u=0, \quad &\text{ on }\partial \Omega.
\end{cases}
\end{equation}

We have the following result.

\begin{theorem}
\label{NP-nc}
Let $d=2,3$. Assume that $K \in \mathcal{C}^1(\mathbb{R})$ such that 
$0<\underline{K}\leq K(s)\leq \overline{K}$ for all $s \in \mathbb{R}$.
Then, we have the following:
\smallskip

\begin{itemize}
\item[$\bullet$] Let $f\in V_0'$. There exists a unique $u\in V_0$ such that
$(K(\theta)\nabla u,\nabla v)=\l f,v\r$, for all $v\in V_0$.
\smallskip

\item[$\bullet$] Let $\theta \in W^{1,r}(\Omega)$, with $d<r< \infty$, and $f\in L^2_0$. Then, $u\in H^2(\Omega)$ and $\partial_\n u=0$ on $\partial \Omega$. Moreover, there exists a positive increasing function $Q$ depending on $\underline{K}$, $\overline{K}$ and $r$ such that, 
$$
\| u\|_{H^2(\Omega)}\leq Q(R) \| f\|,
$$
where $\| \theta\|_{W^{1,r}(\Omega)}\leq R$.
In particular, there exists a constant $C>0$ such that
\begin{align*}
&\| u\|_{H^2(\Omega)}\leq C(1+\| \theta\|_{W^{1,4}(\Omega)})^2 \| f\|,  \quad d=2,\\ 
&\| u\|_{H^2(\Omega)}\leq C(1+\| \theta\|_{W^{1,\infty}(\Omega)})\| f\|, \quad d=2,3.
\end{align*}
\end{itemize}
\end{theorem}
\begin{proof}
The first part is standard. To prove the estimate in $H^2(\Omega)$, 
we follow the argument in \cite[Lemma 4]{A}. We take $v=\frac{w}{K(\theta)}- \overline{\frac{w}{K(\theta)}}$ as test function, with $w\in V$. Note that $v\in V$ since $\nabla \vp\in L^r(\Omega)$ with $r>d$. Hence, for $f\in L^2_0$, we obtain
$$
(\nabla u,\nabla w)= \Big( \frac{K'(\theta)}{K(\theta)} \nabla u \cdot \nabla \theta+ \frac{f}{K(\theta)}, w \Big), \quad \forall \, w \in V.
$$
Since $\| \nabla u\| \leq C \| f\|$ by the first part, we observe that
$$
\Big\| \frac{K'(\theta)}{K(\theta)} \nabla u \cdot \nabla \theta+ \frac{f}{K(\theta)} \Big\|_{L^s(\Omega)}\leq C(1+\| \nabla \theta\|_{L^r(\Omega)})\| f\|,\quad s=\min \lbrace  s_0, 2 \rbrace, \ \text{where} \ \frac{1}{s_0}=\frac12 +\frac{1}{r}.
$$
By the $L^p$-regularity of the Neumann problem with constant coefficients \eqref{NR}, we deduce that $u\in W^{2,s}(\Omega)$. If $s=2$, the proof is complete. Otherwise,  this implies that $u \in W^{1,p}(\Omega)$, where $\frac{1}{p}=\frac{1}{s}-\frac{1}{d}$ and $\| u\|_{W^{1,p}(\Omega)}\leq C(1+\| \nabla \theta\|_{L^r(\Omega)})\| f\|$. Since $p>2$, we can exploit the bound in $W^{1,p}(\Omega)$ to improve the value of $s$. By a finite number of iterations, we eventually find $s=2$ and the bound in $\| u\|_{H^2(\Omega)}\leq Q(R) \| f\|$. The particular case $r=\infty$ can be directly obtained by the above estimated since $s_0=2$, whereas for $r=4$ two iterations of the above argument are sufficient.
\end{proof}
\smallskip

\subsection{Neumann's problem for Laplace operator with logarithmic potential.}
We introduce the homogeneous Neumann problem with a logarithmic convex nonlinear term
\begin{equation}
\label{ELL}
\begin{cases}
-\Delta u+F'(u)=f,\quad &\text{ in }\Omega,\\
\partial_\n u=0, \quad &\text{ on }\partial \Omega.
\end{cases}
\end{equation}
We assume that $F: [-1,1] \mapsto \mathbb{R}$ satisfies $F\in \mathcal{C}([-1,1]) \cap \mathcal{C}^{2}(-1,1)$,
\begin{equation}
\label{F-ass}
\lim_{s\rightarrow -1^+}F^{\prime }(s)=-\infty, \quad
\lim_{s\rightarrow 1^-}F^{\prime }(s)=+\infty, \quad F^{\prime \prime }(s)\geq \theta>0\quad \forall \, s\in (-1,1).
\end{equation}

We now report some existence results and elliptic estimates whose proofs can be found in \cite{A,CG,GMT2018}.

\begin{theorem}
\label{ell2}
Assume that $F$ satisfies the above assumptions.
\smallskip

\begin{itemize}
\item[$\bullet$]
Let $d=2,3$ and $f\in H$. Then, there exists a unique $u\in H^2(\Omega)$ satisfying  $F'(u)\in H$ such that
$-\Delta u+F'(u)=f$ almost everywhere in $\Omega$ and $\partial_{\n} u=0$ almost everywhere on $\partial \Omega$. Moreover, there exists a constant $C>0$ such that
$$
\| u\|_{H^2(\Omega)}+\| F'(u)\|\leq C \big( 1+ \| f\|\big).
$$
In addition, assuming that 
$f_k\rightarrow f$ in $H$, it follows that 
$ u_k\rightarrow u$ in $V$, where $u_k$ and $u$ are the solutions to \eqref{ELL} corresponding to $f_k$ and $f$, respectively.
\smallskip

\item[$\bullet$]
Let $d=2,3$ and $f\in L^p(\Omega)$, where $2\leq p\leq \infty$.
Then, $F'(u)\in L^p(\Omega)$ and
$$
\|F'(u)\|_{L^p(\Omega)}\leq \| f\|_{L^p(\Omega)}.
$$
\item[$\bullet$] Let $d=2,3$ and $f\in V$. 
Then, $u\in W^{2,p}(\Omega)$ where $p=6$ if $d=3$ and for any $p\geq 2$ if $d=2$. We have the estimate
$$
\| \Delta u\|\leq \| \nabla u\|^{\frac12} \| \nabla f\|^{\frac12}.
$$
Moreover, there exists a constant $C>0$, depending on $p$, such that
$$
\|u \|_{W^{2,p}(\Omega)}+ \|F'(u)\|_{L^p(\Omega)}
\leq C \big( 1+\| f\|_{V}\big),
$$
with $p$ as above.
\smallskip

\item[$\bullet$] Let $d=2$ and $f \in V$. Assume that $F$
satisfies
$$
F''(s)\leq e^{C|F'(s)|+C}, \quad \forall \, s \in (-1,1),
$$
for some positive constant $C$.
Then, for any $p\geq 1$, $F''(u)\in L^p(\Omega)$. In addition, there exists a constant $C>0$ (depending on $p$) such that
$$
\| F''(u)\|_{L^p(\Omega)}\leq C\big(1+e^{C\|f\|_V^2}\big).
$$
\end{itemize}
\end{theorem}
\smallskip

\noindent
\textbf{Notation.} 
Throughout the paper, if it is not otherwise stated,
we indicate by $C$ a generic positive
constant depending only on the domain and on structural quantities. The constant $C$ may vary from line to line and
even within the same line. Any further dependence will be explicitly pointed out if necessary.


\section{Assumptions and Existence of Weak Solutions}
\label{S3}
\setcounter{equation}{0}

\noindent
In this paper we address the well-posedness for the HSCH system for a general class of viscosity functions and singular free energy densities. In the sequel we will require the following set of assumptions.

\begin{itemize}
\item[(A1)] The viscosity coefficient $\nu=\nu(s)$
belongs to $\mathcal{C}^2(\mathbb{R})$ and satisfies
\begin{equation}
\label{ipo-nu}
0< \nu_\ast\leq \nu(s)\leq \nu^\ast, \quad \forall \, s\in \mathbb{R}.
\end{equation}

\item[(A2)] The free energy density $\Psi$ can be decomposed into the form
\begin{equation}
\Psi(s)=F(s)-\frac{\theta_0}{2}s^2, \quad \forall \, s \in [-1,1]. \label{psiform}
\end{equation}
The function
$F: [-1,1] \mapsto \mathbb{R}$ satisfies $F\in \mathcal{C}([-1,1]) \cap \mathcal{C}^{4}(-1,1)$,
\begin{align*}
\lim_{s\rightarrow -1^+}&F^{\prime }(s)=-\infty, \quad
\lim_{s\rightarrow 1^-}F^{\prime }(s)=+\infty, \quad F^{\prime \prime }(s)\geq \theta>0\quad \forall \, s\in (-1,1),
\end{align*}
where the positive constants $\theta_0$, $\theta$ satisfy
$\theta_0-\theta:=\alpha>0$.
In addition, there exists $\kappa\in(0,1)$ such that
$$
F^{(3)}(s)s\geq 0, \quad F^{(4)}(s)>0, \quad \forall \, s \in (-1,-1+\kappa] \cup [1-\kappa,1). 
$$
Without loss of generality, we assume $F(0)=F'(0)=0$ and we make the extension that $F(s)=+ \infty$, for all $|s|>1$.

\end{itemize}

In addition to (A2), we will possibly assume the following properties for the free energy density:

\begin{itemize}
\item[(A3)] The second derivative $F''$ is convex and there exists a constant $C>0$ such that
$$
F''(s)\leq C \mathrm{e}^{C|F'(s)|}, \quad \forall \, s \in (-1,1).
$$

\item[(A4)] The potential $\Psi$ has two symmetric minima $\pm \beta$ in $[-1,1]$ such that $\Psi(-\beta)=\Psi(\beta)<0.$
\end{itemize}

\begin{remark}
\label{remarkF}
Assumptions (A2), (A3) and (A4) are motivated by the
logarithmic double-well potential \eqref{SING} with
$F(s)=\frac{\theta}{2}\left[ (1+s)\log(1+s)+(1-s)\log(1-s)\right]$,
for $s \in [-1,1]$. In particular, we observe that there exist $\pm \beta$ which are solutions of the equation $\Psi'(s)=0$ in $[-1,1]$, namely
$$
\frac{\theta}{2} \log\Big(\frac{1+\beta}{1-\beta}\Big)= \theta_0 \beta.
$$
\end{remark}
\smallskip

The first result concerns the existence of global weak solutions in both two and three dimensions.

\begin{theorem}
\label{weak-theorem}
Let $d=2,3$ and $T>0$. Assume that (A1)-(A2) hold.  Given $\vp_0 \in V\cap L^\infty(\Omega)$ such that $\|\vp_0\|_{L^\infty(\Omega)}\leq 1$ and $\overline{\vp}_0\in(-1,1)$,
there exists at least one weak solution 
$(\u, p, \vp)$ to problem \eqref{CHHS}-\eqref{bdini} on $[0,T]$
in the following sense:
\smallskip

\begin{itemize}
\item[$\bullet$] The weak solution $(\uu, p, \vp)$ fulfils the regularity
\begin{align}
&\uu\in L^2(0,T;\mathbf{H}_\sigma),\quad p \in L^q(0,T; V_0),\\
\label{reg-vp-weak}
&\vp \in  \mathcal{C}([0,T],V)\cap L^4(0,T;H^2(\Omega))\cap L^2(0,T;W^{2,p}(\Omega))\cap H^1(0,T;V'),\\
\label{reg-vp-inf}
&\vp \in L^{\infty}(\Omega \times (0,T)) \text{ with } |\vp(x,t)|<1 \text{ a.e. }(x,t) \in \Omega\times(0,T),\\
&\Psi'(\vp)\in L^2(0,T;L^p(\Omega)),
\end{align}
where 
$q=\frac85$ if $d=3$ or any $q\in [1,2)$ if $d=2$, 
$p=6$ if $d=3$ or any $2\leq p < \infty$ if $d=2$.
\smallskip

\item[$\bullet$] The weak solution $(\uu,p,\vp)$ satisfies
 \begin{align}
\label{weak1}
&\nu(\vp)\uu+ \nabla p = \mu \nabla \vp\quad &&\text{ a.e. }(x,t) \in \Omega\times(0,T),\\
\label{weak2}
&\l\partial_t \vp, v\r + ( \uu \cdot \nabla \vp,v )+  (\nabla \mu,\nabla v)=0,
\quad  &&\forall \, v\in V, \ \text{  a.e. } t\in(0,T),
\end{align}
where $\mu=-\Delta \vp+\Psi'(\vp) \in L^2(0,T;V)$. Moreover, $\partial_{\n} \vp=0$ almost everywhere on $\partial \Omega\times(0,T)$ and $\vp(\cdot,0)=\vp_0$ in $\Omega$.
\smallskip

\item[$\bullet$] The weak solution $(\uu,p,\vp)$ fulfils the energy identity
\begin{align}
\ddt \mathcal{E}(\vp(t))+ \|\sqrt{\nu(\vp)}\u(t)\|^2+ \|\nabla \mu(t)\|^2=0,\quad \text{for a.e.\ } t\in (0,T),\label{BEL}
\end{align}
and the mass conservation 
\begin{align}
\int_\Omega \vp(t) \, \d x
=\int_\Omega \vp_0 \, \d x,\quad \forall\, t\in [0,T].\label{mass}
\end{align}
\end{itemize}
\end{theorem}

The proof of Theorem \ref{weak-theorem} can be achieved with the same method exploited in \cite[Theorem 2.1]{GGW2018}, combined with  \cite[Theorem 3.1]{DGL} and \cite[Theorem 2.1]{JWZ} which are valid for regular approximations $\Psi_\varepsilon$ of the potential $\Psi$ (cf. \eqref{defF1s} below). The proof is rather standard and lengthy, and thus it is omitted here.

\begin{remark}
The assumption on the initial total mass $|\overline{\vp}_0|<1$ indicates that the initial datum is allowed to be a mixture but not a single fluid (i.e. $\vp \equiv 1$ or $\vp\equiv -1$).
\end{remark}

\begin{remark}
\label{equiv}
Let $(\uu,p, \vp)$ be a weak solution given by Theorem \ref{weak-theorem}. The pressure $p$ is the weak solution of the Neumann problem
$$
\begin{cases}
- \mathrm{div}\, \Big( \frac{1}{\nu(\vp)} \nabla p \Big) =
- \mathrm{div}\, \Big(\frac{\mu \nabla \vp}{\nu(\vp)}\Big),\quad &\text{ in }\Omega,\\
\partial_\n p=0, \quad &\text{ on }\partial \Omega.
\end{cases}
$$
We note that $\overline{- \mathrm{div}\, (\frac{\mu \nabla \vp}{\nu(\vp)})}=0$ by virtue of the homogeneous Neumann boundary condition for $\vp$ and its regularity. For $v\in V$, multiplying \eqref{weak1} by $\frac{\nabla v}{\nu(\vp)}$ and integrating over $\Omega$, we obtain
$$
\int_\Omega \frac{1}{\nu(\vp)}\nabla p \cdot \nabla v \, \d x
= \int_\Omega \frac{1}{\nu(\vp)}\mu \nabla \vp \cdot \nabla v \, \d x.
$$
On the other hand, multiplying \eqref{weak1} by $\nabla v$, the pressure $p$ can be read as the weak solution of the Neumann problem
$$
\begin{cases}
- \Delta p =
- \mathrm{div}\, \big(-\nu(\vp)\uu+\mu \nabla \vp \big),\quad &\text{ in }\Omega,\\
\partial_\n p=0, \quad &\text{ on }\partial \Omega,
\end{cases}
$$
namely $p= \An \big( - \mathrm{div}\, \big(-\nu(\vp)\uu+\mu \nabla \vp \big)$.  
Since $\vp$ is bounded, using the relation $\mu \nabla \vp= \nabla (\mu \vp)-\vp \nabla \mu$ and integrating by parts, we have
\begin{align*}
|\l - \mathrm{div}\, \big(-\nu(\vp)\uu+\mu \nabla \vp\big), v\r|
&\leq C \big(\| \uu\|+ \| \nabla \mu\|+ \| \mu\|\big) \| \Delta v\|, \quad \forall \, v\in D(A).
\end{align*}
Recalling that $\An: (D(A))'\rightarrow L_0^2$ is a linear and continuous operator, and $\uu \in L^2(0,T;\H_\sigma)$ and $\mu \in L^2(0,T;V)$, the above inequality entails that $p\in L^2(0,T;L_0^2)$.
The Darcy's law \eqref{weak1} can be also rewritten only in terms on $\vp$, which will be needed in the next section. To this purpose, exploiting the identity
\begin{equation}
\label{identity}
\mu \nabla \vp=\nabla\Big(\frac12 |\nabla \vp|^2
+\Psi(\vp) \Big) -\mathrm{div}\, (\nabla \vp \otimes \nabla \vp),
\end{equation}
where $(a \otimes b)_{ij}=a_i b_j$, we rewrite \eqref{weak1} as follows
\begin{equation}
\label{weak1-2}
\nu(\vp)\uu + \nabla p^\ast = - \mathrm{div}\, (\nabla \vp \otimes \nabla \vp),\quad \text{ a.e. }(x,t) \in \Omega\times(0,T),
\end{equation}
where the modified pressure $p^\ast= p- \frac12 |\nabla \vp|^2
+\Psi(\vp)+ \overline{\frac12 |\nabla \vp|^2
+\Psi(\vp)}$. 
It is immediate to infer from assumption (A2) and the regularities  \eqref{reg-vp-weak} and \eqref{reg-vp-inf} that $p^\ast\in L^2(0,T;L_0^2)$. In addition, by \eqref{reg-vp-weak} and \eqref{reg-vp-inf}, it is possible to show (cf. \cite[(3.41)]{GGW2018}) that $\mathrm{div}\, (\nabla \vp \otimes \nabla \vp)\in L^q(0,T;H)$, with $q$ as in Theorem \ref{weak-theorem}, which in turn implies $p^\ast \in L^q(0,T;V)$.
\end{remark}


\section{Uniqueness Results for Weak Solutions in Two Dimensions}
\label{S4}
\setcounter{equation}{0}

In this section we prove two results of uniqueness and continuous dependence for weak solutions in dimension two. First, we prove the uniqueness of solutions to the HSCH with logarithmic potential belonging to a slightly smaller set than weak solutions. More precisely, weak solutions are unique in the class of function satisfying $\vp \in L^\infty(0,T;W^{1,r}(\Omega))$, for some $r>2$ (cf. Theorem \ref{weak-theorem}). This result will be used in the next section to show the propagation of regularity for weak solutions. In addition, thanks to the existence of strong solutions (Theorem \ref{strong2d}), this can be interpreted as a weak-strong uniqueness result. Next, we show the uniqueness of weak solutions to the HSCH system when the logarithmic potential is replaced by its well-known regular (polynomial) approximation.

\subsection{Uniqueness criterion for the HSCH system with logarithmic potential}

\begin{theorem}
\label{weak-strong}
Let $d=2$. Assume that (A1)-(A2) hold. Given $\vp_{01}$,$\vp_{02}$ such that $\vp_{01}\in W^{1,r}(\Omega)$ for some $r>2$, $\vp_{02}\in V$, $\|\vp_{0i}\|_{L^\infty(\Omega)}\leq 1$, $i=1,2$, and $\overline{\vp}_{01}=\overline{\vp}_{02}\in (-1,1)$, consider the two weak solutions $(\uu_1,p_1,\vp_1)$ and $(\uu_2,p_2, \vp_2)$
to \eqref{CHHS}-\eqref{bdini} on $[0,T]$ 
with initial data $\vp_{01}$ and $\vp_{02}$, 
respectively. In addition, suppose that $\vp_1\in L^\infty(0,T; W^{1,r}(\Omega))$. Then, there exists a constant $C>0$ such that
\begin{align*}
\| \vp_1(t)-\vp_2(t)\|_{V_0'}\leq  C\| \vp_{01}-\vp_{02}\|_{V_0'} 
, \quad \forall \, t\in [0,T].
\end{align*}
\end{theorem}

\begin{proof}
Let $(\uu_1,p_1,\vp_1)$ and $(\uu_2,p_2,\vp_2)$ be two global weak solutions to problem \eqref{CHHS}-\eqref{bdini}
on $[0,T]$ with initial data $\vp_{01}$ and $\vp_{02}$, respectively. Setting $ \uu=\uu_1-\uu_2$, $ p^\ast=p_1^\ast-p_2^\ast$ and $ \vp=\vp_1-\vp_2$, we have 
\begin{equation}
\label{diff-problem}
\l \partial_t  \vp, v\r - ( \u_1 \vp, \nabla v )
- (  \u \vp_2,\nabla v ) +
( \nabla  \mu,\nabla v )=0, \quad \forall\, v\in V,
\end{equation}
for almost every $t\in (0,T)$, where $\u$ and $p^\ast$ satisfy  (see \eqref{weak1-2} in Remark \ref{equiv}) 
\begin{equation}
\label{diff-mu-u}
\nu(\vp_1) \uu+ \nabla p^\ast=- \text{div}(\nabla \vp_1 \otimes \nabla \vp)- \text{div}(\nabla \vp \otimes \nabla \vp_2)- (\nu(\vp_1)-\nu(\vp_2))\uu_2,
\end{equation}
and $\mu:=\mu_1-\mu_2$ is given by
$\mu=- \Delta \vp+ \Psi'(\vp_1)-\Psi'(\vp_2) $.
We recall that $\overline{\vp}(t)=0$, for all $t\in (0,T)$. We have the uniform controls (see Theorem \ref{weak-theorem})
\begin{equation}
\label{est-solutions}
\| \vp_i\|_{L^\infty(\Omega \times (0,T))}\leq 1,\quad 
\| \vp_i\|_{L^\infty(0,T;V)}\leq C_0, \ i=1,2, \ \text{ and }\ 
\| \vp_1\|_{L^\infty(0,T;W^{1,r}(\Omega))}\leq R,
\end{equation}
for some $R>0$ due to the regularity condition on $\vp_1$.
Hereafter $C$ will denote a generic constant depending on the parameter of the system, interpolation and embedding results and $C_0$, whereas $C_R$ is a generic constant whose value depend on $R$ in \eqref{est-solutions}.

Taking $v=\An \vp$ in \eqref{diff-problem}, we find the differential equation
$$
\frac12 \ddt \| \vp\|_{V_0'}^2 + (\mu,\vp) = I_1+I_2,
$$
where 
\begin{equation}
\label{Idef}
I_1= (\uu_1 \vp, \nabla \An \vp), \quad 
I_2= (\uu \vp_2, \nabla \An \vp).
\end{equation}
Integrating by parts, and using \eqref{I} and assumption (A2), we obtain
\begin{align}
(\mu,\vp)&= \| \nabla \vp\|^2 + (\Psi'(\vp_1)-\Psi'(\vp_2), \vp)\notag\\
&\geq \| \nabla \vp\|^2- \alpha\| \vp\|^2\notag\\
&\geq \frac12 \| \nabla \vp\|^2-C \| \vp\|_{V_0'}^2.
\label{muvp}
\end{align}
By exploiting \eqref{L}, \eqref{I} and \eqref{N} , we control $I_1$ as follows
\begin{align}
I_1&\leq \| \uu_1\| \| \vp\|_{L^4(\Omega)} \| \nabla \An \vp\|_{L^4(\Omega)}\notag\\
&\leq C \|\uu_1\| \| \nabla \vp\| \|\vp \|_{V_0'} \notag\\
&\leq  \frac{1}{4} \| \nabla \vp\|^2 + C \| \uu_1\|^2 \| \vp\|_{V_0'}^2. \label{I1}
\end{align}
Thanks to \eqref{muvp} and \eqref{I1}, we deduce the differential inequality
\begin{equation}
\label{Idef2}
\frac12 \ddt \| \vp\|_{V_0'}^2 + \frac14 \| \nabla \vp\|^2 
\leq C \big( 1+ \| \uu_1\|^2\big) \| \vp\|_{V_0'}^2 +I_2.
\end{equation}
The rest of the proof is devoted to estimate $I_2$. Since $\nu(\vp_1)$ is strictly positive (cf. (A1)), using the Leray projection operator and \eqref{diff-mu-u}, we rewrite $I_2$ as follows
\begin{align}
I_2&= (\uu, \P(\vp_2 \nabla \An \vp)) \notag \\
&= \Big( \nu(\vp_1)\uu, \frac{1}{\nu(\vp_1)} \P(\vp_2 \nabla \An \vp)\Big)\notag \\
&= -\Big( \nabla p^\ast, \frac{1}{\nu(\vp_1)} \P(\vp_2 \nabla \An \vp) \Big) - \Big( \text{div}(\nabla \vp_1 \otimes \nabla \vp), \frac{1}{\nu(\vp_1)} \P(\vp_2 \nabla \An \vp ) \Big) \notag \\
&\quad - \Big( \text{div}(\nabla \vp \otimes \nabla \vp_2), \frac{1}{\nu(\vp_1)} \P(\vp_2 \nabla \An \vp) \Big) - \Big(\big(\nu(\vp_1)-\nu(\vp_2)\big)\uu_2, \frac{1}{\nu(\vp_1)} \P(\vp_2 \nabla \An \vp) \Big) \notag \\
&= Z_1+Z_2+Z_3+Z_4. \label{I2}
\end{align}
We observe that $Z_1$ can be rewritten as follows
\begin{align*}
Z_1&=-\Big( \frac{1}{\nu(\vp_1)} \nabla p^\ast, \P(\vp_2 \nabla \An \vp) \Big)\\
&= -\Big( \nabla \Big(\frac{p^\ast}{\nu(\vp_1)}\Big), \P(\vp_2 \nabla \An \vp) \Big)- \Big( p^\ast \frac{\nu'(\vp_1)}{\nu^2(\vp_1)}\nabla \vp_1, \P(\vp_2 \nabla \An \vp) \Big)\\
&= - \Big( p^\ast \frac{\nu'(\vp_1)}{\nu^2(\vp_1)}\nabla \vp_1, \P(\vp_2 \nabla \An \vp) \Big).
\end{align*}
Here we have used that $\P\nabla v=0$ for $v\in V$.
By \eqref{O}, \eqref{GN2}, \eqref{I}, \eqref{N} and \eqref{est-solutions}, we have
\begin{align*}
\|\P(\vp_2 \nabla \An \vp) \|_{L^{\frac{2r}{r-2}}(\Omega)}
&\leq C \| \nabla \An \vp\|_{L^\frac{2r}{r-2}(\Omega)}\\
&\leq C \| \vp\|_{V_0'}^{\frac{r-2}{r}} \|\vp\|^{\frac{2}{r}}\\
&\leq C \| \vp\|_{V_0'}^{\frac{r-1}{r}} \| \nabla \vp\|^{\frac{1}{r}}.
\end{align*}
Hence, we get
\begin{equation}
\label{Z1}
Z_1 \leq C \| p^\ast\| \| \nabla \vp_1\|_{L^r(\Omega)}  \| \vp\|_{V_0'}^{\frac{r-1}{r}} \| \nabla \vp\|^{\frac{1}{r}}. 
\end{equation}
In order to find a control of $p^\ast$ in $L^2(\Omega)$, we divide \eqref{diff-mu-u} by $\nu(\vp_1)$ and we test the resulting equation by $\nabla q$, where $q \in H^2(\Omega)$ with $\partial_{\n} q=0$ on $\partial \Omega$. Then, we obtain 
\begin{align}
\Big( \frac{1}{\nu(\vp_1)} \nabla p^\ast, \nabla q \Big)&=
- \Big( \frac{1}{\nu(\vp_1)} \mathrm{div}\, ( \nabla \vp_1 \otimes \nabla \vp ),\nabla q \Big)
- \Big( \frac{1}{\nu(\vp_1)} \mathrm{div}\, ( \nabla \vp \otimes \nabla \vp_2 ),\nabla q \Big) \notag \\
&\quad -\Big(  \frac{\nu(\vp_1)-\nu(\vp_2)}{\nu(\vp_1)} \uu_2, \nabla q\Big). \label{pi1}
\end{align} 
Due to the boundary condition satisfied by $q$, after an integration by parts, we have
$$
\Big( \frac{1}{\nu(\vp_1)}\nabla p^\ast , \nabla q \Big)=-\Big( p^\ast, \mathrm{div}\, \Big( \frac{\nabla q}{\nu(\vp_1)}\Big)\Big).
$$
For a $d\times d$ tensor $S$ and two vector fields $\vv$ and $\ww$, we recall the relations 
$\mathrm{div}\, \big( S^t \vv\big)= S^t : \nabla \vv + \mathrm{div}\, S \cdot \vv$, where $A:B= \sum_{i,j=1}^d A_{ij}B_{ij}$, and $(\vv\otimes \ww)^t=\ww\otimes \vv$. Accordingly, we rewrite the first two terms on the right-hand side of \eqref{pi1} as follows
\begin{align*}
&- \int_{\Omega}  \mathrm{div}\, ( \nabla \vp_1 \otimes \nabla \vp ) \cdot \frac{\nabla q}{\nu(\vp_1)} \, \d x
- \int_{\Omega}  \mathrm{div}\, ( \nabla \vp \otimes \nabla \vp_2 ) \cdot \frac{\nabla q}{\nu(\vp_1)} \, \d x \\
&= - \int_{\Omega} \mathrm{div} \, \Big( (\nabla \vp_1 \otimes \nabla \vp)^t \frac{\nabla q}{\nu(\vp_1)} \Big)\, \d x+
\int_{\Omega} \nabla \vp_1 \otimes \nabla \vp : \nabla \Big( \frac{\nabla q}{\nu(\vp_1)} \Big) \, \d x\\
& \quad - \int_{\Omega} \mathrm{div} \, \Big( (\nabla \vp \otimes \nabla \vp_2)^t \frac{\nabla q}{\nu(\vp_1)} \Big)\, \d x+
\int_{\Omega} \nabla \vp \otimes \nabla \vp_2 : \nabla \Big( \frac{\nabla q}{\nu(\vp_1)} \Big) \, \d x\\
&= - \int_{\partial \Omega} \frac{1}{\nu(\vp_1)} 
\nabla \vp \otimes \nabla \vp_1 \nabla q \cdot \n \, \d \sigma 
+\int_{\Omega} 
\nabla \vp_1 \otimes \nabla \vp : \nabla \Big( \frac{\nabla q}{\nu(\vp_1)} \Big) \, \d x\\
&\quad  - \int_{\partial \Omega} \frac{1}{\nu(\vp_1)} 
\nabla \vp_2 \otimes \nabla \vp \nabla q \cdot \n \, \d \sigma 
+\int_{\Omega} \nabla \vp \otimes \nabla \vp_2 : \nabla \Big( \frac{\nabla q}{\nu(\vp_1)} \Big) \, \d x\\
&= - \int_{\partial \Omega} \frac{1}{\nu(\vp_1)} 
 \Big( \nabla \vp_1 \cdot \nabla q\Big) 
 \Big( \nabla \vp \cdot \n \Big) \, \d \sigma 
+\Big(\nabla \vp_1 \otimes \nabla \vp, \nabla \Big( \frac{\nabla q}{\nu(\vp_1)}  \Big) \Big)\\
&\quad  - \int_{\partial \Omega} \frac{1}{\nu(\vp_1)} 
\Big( \nabla \vp \cdot \nabla q \Big) \Big( \nabla \vp_2 \cdot \n \Big) \, \d \sigma 
+\Big(\nabla \vp_1 \otimes \nabla \vp, \nabla \Big( \frac{\nabla q}{\nu(\vp_1)} \Big) \Big)\\
&= \Big(\nabla \vp_1 \otimes \nabla \vp, \nabla \Big( \frac{\nabla q}{\nu(\vp_1)} \Big) \Big)+
\Big(\nabla \vp \otimes \nabla \vp_2, \nabla \Big( \frac{\nabla q}{\nu(\vp_1)} \Big) \Big).
\end{align*}
Thus, combining the two expressions above with \eqref{pi1}, we deduce that
\begin{align}
-\Big( p^\ast, \mathrm{div}\, \Big( \frac{\nabla q}{\nu(\vp_1)}\Big)\Big)&= \Big(\nabla \vp_1 \otimes \nabla \vp + \nabla \vp \otimes \nabla \vp_2, \nabla \Big( \frac{\nabla q}{\nu(\vp_1)} \Big) \Big) \notag\\
&\quad -\Big(  \frac{\nu(\vp_1)-\nu(\vp_2)}{\nu(\vp_1)} \uu_2, \nabla q\Big). \label{pi2}
\end{align}
Next, we choose the test function $q$ such that 
\begin{equation}
\label{q-def}
\begin{cases}
-\mathrm{div}\, \big( \frac{1}{\nu(\vp_1)} \nabla q \big) = p^\ast, \quad &\text{ in }\Omega ,\\
\partial_\n q=0, \quad &\text{ on }\partial \Omega.
\end{cases}
\end{equation}
Since $p^\ast \in L^2(0,T;L_0^2)$ and $\vp_1\in L^\infty(0,T;W^{1,r}(\Omega))$, Theorem \ref{NP-nc} entails that there exists a unique $q\in L^2(0,T;H^2(\Omega))$ which satisfies \eqref{q-def}. 
In particular, we have the following estimate
\begin{equation}
\label{estqH2}
\| q\|_{H^2(\Omega)}\leq C_R \| p^\ast\|.
\end{equation}
 By definition of $q$, we can rewrite \eqref{pi2} as
\begin{align}
\| p^\ast\|^2 &= \Big(\nabla \vp_1 \otimes \nabla \vp+\nabla \vp \otimes \nabla \vp_2, \nabla \Big( \frac{1}{\nu(\vp_1)} \nabla q \Big) \Big) -\Big(  \frac{\nu(\vp_1)-\nu(\vp_2)}{\nu(\vp_1)} \uu_2, \nabla q\Big) \notag\\
&= \Big(\nabla \vp_1 \otimes \nabla \vp+\nabla \vp \otimes \nabla \vp_2, \frac{\nu'(\vp_1)}{\nu(\vp_1)^2} \nabla \vp_1 \otimes \nabla q + \frac{1}{\nu(\vp_1)} \nabla \nabla q \Big) \notag \\
&\quad -\Big(  \frac{\nu(\vp_1)-\nu(\vp_2)}{\nu(\vp_1)} \uu_2, \nabla q\Big).
\label{pL2}
\end{align}
By using assumption (A1), the embedding $V\hookrightarrow L^p(\Omega)$, for any $p\geq 1$, together with \eqref{GN2}, \eqref{I} and \eqref{estqH2}, we control the terms on right-hand side as follows
\begin{align*}
\Big(\nabla \vp_1 \otimes \nabla \vp+&\nabla \vp \otimes \nabla \vp_2, \frac{\nu'(\vp_1)}{\nu(\vp_1)^2} \nabla \vp_1 \otimes \nabla q + \frac{1}{\nu(\vp_1)} \nabla \nabla q \Big) \\
&\leq C \Big(\| \nabla \vp_1\|_{L^\infty(\Omega)}+ \|\nabla \vp_2 \|_{L^\infty(\Omega)}\Big)  \Big(1+\| \nabla \vp_1\|_{L^r(\Omega)}\Big) \|\nabla \vp \| \| q\|_{H^2(\Omega)}\\
&\leq C_R \Big(\| \nabla \vp_1\|_{L^\infty(\Omega)}+ \|\nabla \vp_2 \|_{L^\infty(\Omega)}\Big) \|\nabla \vp \| \| p^\ast\|,
\end{align*}
and
\begin{align*}
-\Big(  \frac{\nu(\vp_1)-\nu(\vp_2)}{\nu(\vp_1)} \uu_2, \nabla q\Big) &\leq C \| \uu_2\| \| \vp \|_{L^r(\Omega)}  \| q\|_{H^2(\Omega)}\\
&\leq 
C_R \| \uu_2 \|\| \vp\|^{\frac{2}{r}}\| \nabla \vp\|^{\frac{r-2}{r}}\| p^\ast\|\\
&\leq C_R \| \uu_2 \|\| \vp\|_{V_0'}^{\frac{1}{r}}\| \nabla \vp\|^{\frac{r-1}{r}}\| p^\ast\|.
\end{align*}
Hence, we obtain the following estimate for the pressure
\begin{equation}
\label{pressure}
\| p^\ast \| \leq C_R \big(\| \nabla \vp_1\|_{L^\infty(\Omega)}+ \|\nabla \vp_2 \|_{L^\infty(\Omega)}\big) 
 \|\nabla \vp \|
+C_R \| \u_2\| \| \vp \|_{V_0'}^{\frac{1}{r}} \| \nabla \vp\|^{\frac{r-1}{r}}.
\end{equation}
Combining \eqref{Z1} with \eqref{pressure}, we deduce that
\begin{align*}
Z_1&\leq 
C_R \big(\| \nabla \vp_1\|_{L^\infty(\Omega)}+ \|\nabla \vp_2 \|_{L^\infty(\Omega)}\big) \| \vp\|_{V_0'}^{\frac{r-1}{r}} \| \nabla \vp\|^{\frac{r+1}{r}}+
C_R \| \u_2\| \| \vp\|_{V_0'} \| \nabla \vp\|.
\end{align*}
Then, by Young's inequality
\begin{align}
Z_1&\leq 
C_R \big(\| \nabla \vp_1\|_{L^\infty(\Omega)}+ \|\nabla \vp_2 \|_{L^\infty(\Omega)}\big) \| \vp\|_{V_0'}^{\frac{r-1}{r}} \| \nabla \vp\|^{\frac{r+1}{r}}+
C_R \| \u_2\| \| \vp\|_{V_0'} \| \nabla \vp\| \notag \\
&\leq \frac{1}{32} \| \nabla \vp\|^2+ C_R \Big(\| \nabla \vp_1\|_{L^\infty(\Omega)}^{\frac{2r}{r-1}}+ \|\nabla \vp_2 \|_{L^\infty(\Omega)}^{\frac{2r}{r-1}}+ \| \u_2\|^2 \Big) \| \vp\|_{V_0'}^2. \label{Z1final}
\end{align}
We now proceed to estimate $Z_2$ and $Z_3$.
Using integration by parts, we have
\begin{align}
Z_2&= - \int_{\Omega} \mathrm{div}\, \Big( (\nabla \vp_1\otimes \nabla \vp)^t \frac{\P(\vp_2 \nabla \An \vp)}{\nu(\vp_1)} \Big)
\, \d x+ \int_{\Omega} \nabla \vp_1\otimes \nabla \vp : \nabla \Big( \frac{ \P(\vp_2\nabla \An \vp)}{\nu(\vp_1)} \Big) \, \d x \notag \\
&=- \int_{\partial \Omega} \nabla \vp\otimes \nabla \vp_1 \frac{\P(\vp_2 \nabla \An \vp)}{\nu(\vp_1)} \cdot \n \, \d \sigma +
\int_{\Omega} \nabla \vp_1\otimes \nabla \vp : \nabla \Big( \frac{ \P(\vp_2\nabla \An \vp)}{\nu(\vp_1)} \Big) \, \d x \notag \\
&=- \int_{\partial \Omega} \Big( \nabla \vp_1 \cdot \frac{\P(\vp_2\nabla \An \vp)}{\nu(\vp_1)}\Big) \Big(\nabla \vp \cdot \n\Big) \, \d \sigma +
\Big( \nabla \vp_1\otimes \nabla \vp, \nabla \Big( \frac{\P(\vp_2 \nabla \An \vp)}{\nu(\vp_1)}\Big) \Big) \notag \\
&= \Big( \nabla \vp_1\otimes \nabla \vp, 
\frac{ \nabla\P(\vp_2 \nabla \An \vp)}{\nu(\vp_1)}
+ \P(\vp_2 \nabla \An \vp) \otimes \nabla \frac{1}{\nu(\vp_1)} \Big).
\label{Z2-def}
\end{align}
Similarly, we find
\begin{equation}
\label{Z3-def}
Z_3= \Big( \nabla \vp\otimes \nabla \vp_2,  \frac{ \nabla\P(\vp_2 \nabla \An \vp)}{\nu(\vp_1)}
+ \P(\vp_2 \nabla \An \vp) \otimes \nabla \frac{1}{\nu(\vp_1)}\Big).
\end{equation}
By exploiting (A1), \eqref{O}, \eqref{L}, \eqref{Ad2} and \eqref{est-solutions}, we obtain
\begin{align}
Z_2&\leq C\| \nabla \vp_1\|_{L^4(\Omega)}
\| \nabla \vp\| \Big(
\Big\|  \frac{\nabla\P(\vp_2 \nabla \An \vp)}{\nu(\vp_1)} \Big\|_{L^{4}(\Omega)} + 
\Big\|\P(\vp_2 \nabla \An \vp) \otimes \nabla \frac{1}{\nu(\vp_1)} \Big\|_{L^4(\Omega)}
\Big)\nonumber\\
&\leq C
\| \vp_1\|_{H^2(\Omega)}^{\frac12}
\| \nabla \vp\| \times \notag \\
&\quad \Big(
\| \nabla \P(\vp_2 \nabla \An \vp)
\|^{\frac12}
\| \nabla \P(\vp_2 \nabla \An \vp)
\|_{V}^{\frac12}+
\|\P(\vp_2 \nabla \An \vp) \|_{L^\infty(\Omega)}
\Big\| \frac{\nu'(\vp_1)}{\nu(\vp_1)^2}\nabla \vp_1\Big\|_{L^4(\Omega)}\Big)
\nonumber\\
&\leq C \| \vp_1\|_{H^2(\Omega)}^{\frac12}
\| \nabla \vp\| \times \notag \\
&\Big(
\|\vp_2 \nabla \An \vp\|_V^{\frac12}
\| \vp_2 \nabla \An \vp\|_{H^2(\Omega)}^{\frac12}+
\|\vp_2 \nabla \An \vp \|^{\frac12}
\|\vp_2 \nabla \An \vp \|_{H^2(\Omega)}^{\frac12}
\| \vp_1\|_{H^2(\Omega)}^{\frac12} \Big)
\label{Z2-1}
\end{align}
By using \eqref{Ad2}, \eqref{prodV}, \eqref{prodH2}, \eqref{I}, \eqref{N} and \eqref{est-solutions}, we preliminary infer that
\begin{align}
\| \vp_2 \nabla \An \vp\|_{V}
&\leq C \| \vp_2\|_{V} \| \nabla \An \vp\|_{L^\infty(\Omega)}+ C \| \vp_2\|_{L^\infty(\Omega)} \| \nabla \An \vp\|_V \notag\\
&\leq C \| \vp\|_{V_0'}^{\frac12} \| \nabla \vp\|^{\frac12}.
\label{P-V} 
\end{align}
and
\begin{align}
\| \vp_2 \nabla \An \vp
\|_{H^2(\Omega)}
&\leq C \| \vp_2\|_{H^2(\Omega)} \| \nabla \An \vp\|_{L^\infty(\Omega)}+
C \| \vp_2\|_{L^\infty(\Omega)}\| \nabla \An \vp\|_{H^2(\Omega)}\notag\\
&\leq C \| \vp_2\|_{H^2} \| \vp\|_{V_0'}^\frac12 \| \nabla \vp\|^{\frac12}+ C\| \nabla \vp\|.
\label{P-H2}
\end{align}
Combining \eqref{Z2-1} with \eqref{P-V} and \eqref{P-H2}, and then using the Young's inequality, we obtain
\begin{align}
Z_2&\leq
C \big( \| \vp_1\|_{H^2(\Omega)}^{\frac12}
\| \vp_2\|_{H^2(\Omega)}^\frac12 + \|\vp_1\|_{H^2(\Omega)} \big) \| \vp\|_{V_0'}^{\frac12}
\| \nabla \vp\|^{\frac32}
+C \| \vp_1\|_{H^2(\Omega)}^{\frac12}
\| \vp\|_{V_0'}^{\frac14}\| \nabla \vp\|^{\frac74} \notag\\
&\quad +\|\vp_1\|_{H^2(\Omega)} \| \vp_2\|_{H^2(\Omega)}^\frac12 
\| \vp\|_{V_0'}^\frac32 \| \nabla \vp\|^\frac54\notag\\
&\leq \frac{1}{32}\| \nabla \vp\|^2+
C \big( \| \vp_1\|_{H^2(\Omega)}^4+ \| \vp_2\|_{H^2(\Omega)}^4 \big)\| \vp\|_{V_0'}^2.
\label{Z2-final}
\end{align}
By the same argument, we easily deduce that
\begin{equation}
Z_3\leq \frac{1}{32} \| \nabla \vp\|^2 + 
C \big( \| \vp_1\|_{H^2(\Omega)}^4+\| \vp_2\|_{H^2(\Omega)}^4
\big) \| \vp\|_{V_0'}^2. 
\label{Z3-final}
\end{equation}
We now control $Z_4$.  By (A1), \eqref{I}, \eqref{L} and \eqref{est-solutions}, we have
\begin{align}
Z_4 &\leq C  \| \uu_2\| \| \vp\|_{L^4(\Omega)} 
\| \P (\vp_2 \nabla \An \vp) \|_{L^4(\Omega)}\notag\\
&\leq C \| \uu_2\| \| \vp\|^\frac12 \|\nabla \vp\|^\frac12 
\| \nabla \N \vp\|^\frac12
\| \nabla \N \vp\|_{V}^\frac12\notag\\
&\leq C \| \uu_2\| \| \vp\|_{V_0'} \| \nabla \vp\|\notag\\
&\leq \frac{1}{32} \|\nabla \vp\|^2 + C \| \uu_2\|^2 \| \vp\|_{V_0'}^2.
\label{Z4-final}
\end{align}
Collecting \eqref{Z1final}, \eqref{Z2-final}, \eqref{Z3-final} and \eqref{Z4-final}, we eventually infer that
$$
I_2\leq \frac18 \| \nabla \vp\|^2+
C  \Big( \| \uu_2\|^2 + 
\| \vp_1\|_{H^2(\Omega)}^4+
\| \vp_2\|_{H^2(\Omega)}^4  +
+ \| \nabla \vp_1\|_{L^\infty(\Omega)}^{\frac{2r}{r-1}}+ 
\| \nabla \vp_2\|_{L^\infty(\Omega)}^{\frac{2r}{r-1}} \Big) \|\vp \|_{V_0'}^2.
$$
Finally, we arrive at the differential inequality
\begin{equation}
\label{diffinequniq}
\frac12 \ddt \| \vp\|_{V_0'}^2 + \frac18 \| \nabla \vp\|^2 
\leq \Lambda \| \vp\|_{V_0'}^2,
\end{equation}
where
$$
\Lambda=C_R \Big( 1+ \| \uu_1\|^2 + \| \uu_2\|^2 
+\| \vp_1\|_{H^2(\Omega)}^4
+\| \vp_2\|_{H^2(\Omega)}^4
+ \| \nabla \vp_1\|_{L^\infty(\Omega)}^{\frac{2r}{r-1}}+ 
\| \nabla \vp_2\|_{L^\infty(\Omega)}^{\frac{2r}{r-1}}\Big).
$$
In order to conclude the proof via the Gronwall lemma, we need to show that $\Lambda\in L^1(0,T)$. In light of the regularity of weak solutions (cf. Theorem \ref{weak-theorem}), we are only left to prove that $\nabla \vp_i \in L^{\frac{2r}{r-1}}(0,T;L^\infty(\Omega))$, $i=1,2$. We recall that, for any $\sigma \in (0,1)$,
$\log(1+Cs)\leq (1+Cs)^\sigma$, for all $s\geq 0$.
Taking $\sigma= \frac{r-2}{r}$ in this inequality, we apply the 
Br\'{e}zis-Gallouet-Wainger inequality \eqref{BWd2} for some $p>2$ and the Young's inequality
\begin{align*}
 \| \nabla \vp_i\|_{L^\infty(\Omega)}^{\frac{2r}{r-1}}
 &\leq C \| \vp_i\|_{H^2(\Omega)}^{\frac{2r}{r-1}} \log(1+C\| \vp_i\|_{W^{2,p}(\Omega)})^\frac{r}{r-1}\\
 &\leq C \| \vp_i\|_{H^2(\Omega)}^4+ C \log(1+C\| \vp_i\|_{W^{2,p}(\Omega)})^\frac{4r}{2r-4}\\
 &\leq C \| \vp_i\|_{H^2(\Omega)}^4+ C (1+ \| \vp_i\|_{W^{2,p}(\Omega)}^2.
\end{align*}
Here we have used that $\frac{2r}{r-1}<4$ which follows from $r>2$. Thus, $\Lambda \leq \widetilde{\Lambda}$, where 
$$
\widetilde{\Lambda}\leq C_R \Big( 1+ \| \uu_1\|^2 + \| \uu_2\|^2 
+\| \vp_1\|_{H^2(\Omega)}^4
+\| \vp_2\|_{H^2(\Omega)}^4 
+ \| \vp_1\|_{W^{2,p}(\Omega)}^2
+ \| \vp_2\|_{W^{2,p}(\Omega)}^2
\Big) 
$$
and $\widetilde{\Lambda} \in L^1(0,T)$ by Theorem \ref{weak-theorem}.
An application of the Gronwall lemma to \eqref{diffinequniq} gives us
$$
\| \vp_1(t)-\vp_2(t)\|^2_{V_0'}\leq \| \vp_1(0)-\vp_2(0)\|^2_{V_0'} \mathrm{e}^{C_R \int_0^t \widetilde{\Lambda}(\tau) \, \d \tau},
$$
which implies the desired conclusion.
\end{proof}

\begin{remark}
The additional regularity condition $\vp_1\in L^\infty(0,T;W^{1,r}(\Omega))$, for some $r>2$,
has been only used to control $Z_1$. In particular, it played a crucial role to deduce the estimate \eqref{estqH2}. We observe that different criteria involving the $L^p(0,T;H^2(\Omega))$-norm of $\vp_1$ (with $p>4$ large) might be formulated in accordance with Theorem \ref{NP-nc} to control the $H^2$-norm of $q$ in \eqref{q-def}. However, it remains an open question whether an argument similar to the one employed in \cite[Theorem 3.1]{GMT2018} can be adapted for the HSCH system. 
\end{remark}


\subsection{Uniqueness for the HSCH system with regular potentials}

\noindent
Let us consider the HSCH system with polynomial (regular) potential $\Psi_0(s)= (s^2-1)^2$ for all $s\in \mathbb{R}$. We report a result concerning the existence of global weak solution proved in \cite[Theorem 3.1]{DGL}.

\begin{theorem}
\label{existence-regular}
Let $d=2$ and $\vp_0 \in V$. Assume that (A1) holds. For any $T>0$, there exists at least one weak solution $(\u, p, \vp)$ to problem \eqref{CHHS}-\eqref{bdini} with $\Psi_0(s)= (s^2-1)^2$ on $[0,T]$ such that 
\begin{align*}
&\u \in L^2(0,T; \H_\sigma), \quad p \in L^q(0,T;V_0),\\
&\vp \in \mathcal{C}([0,T],V)\cap L^2(0,T;H^3(\Omega))\cap W^{1,q}(0,T;V'),\\
&\mu\in L^2(0,T;V),
\end{align*}
for any $\frac43\leq q<2$, and satisfies
\begin{align*}
&\nu(\vp)\uu+ \nabla p = \mu \nabla \vp, \quad \mu=-\Delta\vp+\Psi_0'(\vp), \quad &&\text{ a.e. }(x,t) \in \Omega\times(0,T),\\
&\l\partial_t \vp, v\r + ( \uu \cdot \nabla \vp,v )+  (\nabla \mu,\nabla v)=0,
\quad  && \forall \, v\in V, \ \text{  a.e. } t\in(0,T).
\end{align*}
Moreover, $\partial_\n \vp=0$ almost everywhere on $\partial \Omega \times (0,T)$ and $\vp(\cdot,0)=\vp_0$ in $\Omega$.
\end{theorem}

\noindent
Two differences must be pointed out between the polynomial potential case and the logarithmic potential case. First, the highest regularity in space in the former case is $\vp \in L^2(0,T;H^3(\Omega))$, whereas in the latter case we only reach $\vp \in L^2(0,T;W^{2,p}(\Omega))$, for any $p\geq 2$ (in two dimensions). 
Second, and more importantly, we can control the second derivative $\Psi_0''(\vp)$ in terms of $L^p$-norms of $\vp$. Thus, 
it is possible to control the difference of two solutions in $L^2$, in contrast to the estimate in the dual space $V'$ for the logarithmic potential case.   

\begin{theorem}
\label{RU}
Let $d=2$ and $\vp_{01}, \vp_{02}$ be such that $\vp_{0i}\in V$, $i=1,2$, and $\overline{\vp}_{01}=\overline{\vp}_{02}$. Assume that 
$(\u_1,p_1,\vp_1)$ and $(\u_2,p_2,\vp_2)$ are two solutions given by Theorem \ref{existence-regular} with initial data 
$\vp_{01}$ and $\vp_{02}$, respectively. Then, there exists a positive constant $C=C(T)$ such that
$$
\| \vp_1(t)-\vp_2(t)\|\leq C \| \vp_{01}-\vp_{02}\|, \quad 
\forall \, t\in [0,T]. 
$$
In particular, the weak solution to system \eqref{CHHS}-\eqref{bdini} (with polynomial potential) is unique.
\end{theorem}

\begin{proof}
Let us consider the difference of two solutions $\uu=\uu_1-\uu_2$, $ p=p_1-p_2$ and $ \vp=\vp_1-\vp_2$. We have
\begin{equation}
\label{diff-problem-reg}
\l \partial_t  \vp, v\r - ( \u_1 \vp, \nabla v )
- (  \u \vp_2,\nabla v ) +
( \nabla  \mu,\nabla v )=0, \quad \forall\, v\in V,
\end{equation}
for almost every $t\in (0,T)$, where
\begin{align}
\label{diff-u-reg}
\nu(\vp_1) \uu+ \nabla \widetilde{p}=- \Delta \vp_1 \nabla \vp-\Delta \vp \nabla \vp_2- (\nu(\vp_1)-\nu(\vp_2))\uu_2,
\end{align}
and 
\begin{align}
\label{diff-mu-reg}
&\mu= - \Delta \vp+ \Psi_0'(\vp_1)-\Psi_0'(\vp_2).
\end{align}
Here $\widetilde{p}=p+\Psi_0(\vp)$ and $\mu=\mu_1-\mu_2$. According to Remark \ref{equiv}, it is easily seen that $\widetilde{p}$ has the same regularity properties of $p$ and $p^\ast$. We observe that
$\overline{\vp}=0$ for all $t\in [0,T]$. Moreover, we have 
\begin{equation}
\label{est-sol-reg}
\| \vp_i\|_{L^\infty(0,T;V)}\leq C_0, \quad i=1,2.
\end{equation}
Taking $v=\vp$ in \eqref{diff-problem-reg}, and using the chain rule in $L^{q'}(0,T;V)\cap W^{1,q}(0,T;V)$, where $\frac{1}{q'}+\frac{1}{q}=1$, we obtain
$$
\frac12 \ddt \| \vp\|^2 + (\nabla \mu, \nabla \vp)= J_1+J_2,
$$
having set
$$
J_1= (\u_1 \vp, \nabla \vp),\quad 
J_2=(\u \vp_2, \nabla \vp).
$$
We report the basic estimates
\begin{equation}
\label{H1-inter}
\| \nabla \vp\|\leq \| \vp\|^\frac12 \|\Delta \vp \|^\frac12,\quad \| \vp\|_{H^2(\Omega)}\leq C\| \Delta \vp\|.
\end{equation}
By using the Sobolev embedding $V\hookrightarrow L^6(\Omega)$, the form of $\Psi_0$, the estimates \eqref{est-sol-reg} and \eqref{H1-inter},  and Young's inequality, we have
\begin{align*}
(\nabla \mu, \nabla \vp)&= \| \Delta \vp\|^2- (\Psi_0'(\vp_1)-\Psi_0'(\vp_2), \Delta \vp)\\
&\geq \| \Delta \vp\|^2- C \big( \| \Psi_0''(\vp_1)\|_{L^3(\Omega)}+ \| \Psi_0''(\vp_1)\|_{L^3(\Omega)}\big) \| \vp\|_{L^6(\Omega)} \| \Delta \vp\|\\
&\geq \| \Delta \vp\|^2- C \| \vp\|^\frac12 \| \Delta \vp\|^\frac32\\
&\geq \frac12 \| \Delta \vp\|^2-C \| \vp\|^2.
\end{align*}
By \eqref{L}, \eqref{Ad2} and \eqref{H1-inter}, we deduce that
\begin{align*}
J_1&\leq \| \u_1 \| \| \vp\|_{L^4(\Omega)}
\| \nabla \vp\|_{L^4(\Omega)}\\
&\leq C \| \u_1 \| \|\vp \| \| \Delta \vp\|  \\
&\leq \frac18 \| \Delta \vp\|^2+ C \| \u_1\|^2 \| \vp\|^2,
\end{align*}
and 
\begin{align*}
J_2&= -(\uu \cdot \nabla \vp_2, \vp) \\
&\leq C\| \u\| \| \vp\|_{L^\infty(\Omega)}\\
&\leq C \| \u\| \| \vp\|^\frac12 \| \Delta \vp\|^\frac12.
\end{align*}
Now we multiply \eqref{diff-u-reg} by $\u$ and we integrate over $\Omega$. Noticing that 
$(-\Delta \vp_1 \nabla \vp,\u)= (\vp \nabla \Delta \vp_1, \u)$, we obtain
\begin{align*}
\nu_\ast \| \u\|^2 &\leq (\vp \nabla \Delta \vp_1, \u)
-(\Delta \vp \nabla \vp_2, \u)- ((\nu(\vp_1)-\nu(\vp_2))\u_2, \u)\\
&\leq \| \vp\|_{L^\infty(\Omega)} \|\nabla \Delta \vp_1 \| \| \u\| 
+ \| \Delta \vp\| \| \nabla \vp_2\|_{L^\infty(\Omega)} \| \u\|+ C \| \vp\|_{L^\infty(\Omega)}
\|\u_2\| \| \u\|.
\end{align*}
Hence, by \eqref{Ad2} and \eqref{est-sol-reg}, we eventually find
$$
\| \u\|\leq C \Big( \| \vp_1\|_{H^3(\Omega)}+ \| \u_2\| \Big) \| \vp\|^\frac12 \| \Delta \vp\|^\frac12+ C \| \vp_2\|_{H^3(\Omega)}^\frac12 
\| \Delta \vp\|. 
$$
Combining the above estimates and using Young's inequality, we deduce that
\begin{align*}
J_2&\leq C \Big( \| \vp_1\|_{H^3(\Omega)}+ \| \u_2\| \Big) \| \vp\|\| \Delta \vp\|+ C \| \vp_2\|_{H^3(\Omega)}^\frac12 
\| \vp\|^\frac12 \| \Delta \vp\|^\frac32\\
&\leq \frac18 \| \Delta \vp\|^2 +C \Big( \| \vp_1\|_{H^3(\Omega)}^2+  \| \vp_2\|_{H^3(\Omega)}^2+\| \u_2\|^2 \Big) \| \vp\|^2.
\end{align*} 
We finally end up with the differential inequality
$$
\frac12 \ddt \| \vp\|^2 +  \frac14 \| \Delta \vp\|^2 \leq C
\Big( 1+\| \vp_1\|_{H^3(\Omega)}^2+  \| \vp_2\|_{H^3(\Omega)}^2
+\| \u_1\|^2 +\| \u_2\|^2 \Big) \| \vp\|^2. 
$$ 
In light of the regularity $\vp_i \in L^2(0,T;H^3(\Omega))$, 
$\u_i \in L^2(0,T;\mathbf{H}_\sigma)$, for $i=1,2$, the claim easily follows from the Gronwall lemma. 
\end{proof}


\section{Global Strong Solutions in Two Dimensions}
\label{S5}
\setcounter{equation}{0}

This section is devoted to the existence and uniqueness of global strong solutions for the HSCH system with unmatched viscosities and logarithmic potential in dimension two. 

\begin{theorem}
\label{strong2d}
Let $d=2$, $\vp_0 \in H^2(\Omega)$ such that $\| \vp_0\|_{L^\infty(\Omega)}\leq 1$, $\overline{\vp}_0=m \in (-1,1)$, $\widetilde{\mu}_0=-\Delta \vp_0+F'(\vp_0) \in V$ and $\partial_\n \vp_0=0$ on $\partial \Omega$. Assume that (A1)-(A3) hold.
For any $T>0$, there exists a unique strong solution to \eqref{CHHS}-\eqref{bdini} on $[0,T]$ such that 
\begin{align*}
&\uu \in L^\infty(0,T;\H_\sigma\cap \mathbf{V}), \quad p\in L^\infty(0,T;H^2(\Omega))\\
&\vp \in L^\infty(0,T;W^{2,p}(\Omega))\cap H^1(0,T;V),\\
&\vp \in L^{\infty}(\Omega \times (0,T)) \text{ with } |\vp(x,t)|<1 \text{ a.e. }(x,t) \in \Omega\times(0,T),\\
&\mu \in L^{\infty}(0,T;V)\cap L^2(0,T;H^3(\Omega))\cap H^1(0,T;V'),\\
&\Psi''(\vp)\in L^\infty(0,T;L^p(\Omega)),
\end{align*}
for any $2\leq p<\infty$. The strong solution satisfies \eqref{CHHS} for almost every $(x,t)\in \Omega \times (0,T)$ and \eqref{bdini} for almost every $(x,t)\in \partial \Omega \times (0,T)$, and $\vp(\cdot,0)=\vp_0(\cdot)$ in $\Omega$.
In addition, given two strong solutions $(\uu_1,p_1,\vp_1)$ and $(\uu_2,p_2,\vp_2)$ on $[0,T]$ with initial data $\vp_{01}$ and $\vp_{02}$ satisfying the above assumptions, the following continuous dependence estimate holds
\begin{equation}
\label{CDstrong}
\| \vp_1(t)-\vp_2(t)\|\leq C \| \vp_{01}-\vp_{02}\|, \quad \forall \, t \in [0,T],
\end{equation}
where the constant $C>0$  depends on $T$ and the norms of the initial data. 
\end{theorem}


\begin{proof}
The proof is carried out in several steps.
\medskip

\textbf{Step 1: Family of approximating regular potentials.}
First, we define a family of regular functions $ \lbrace \Psi_\varepsilon \rbrace$ which are defined on $\mathbb{R}$. For any $\varepsilon \in (0,1]$, we introduce 
$$
\Psi_\varepsilon(s) = F_\varepsilon (s)-\frac{\theta_0}{2}s^2, \quad \forall \, s \in \mathbb{R},
$$
where
\begin{equation}
F_\varepsilon(s)=
\begin{cases}
 \displaystyle{\sum_{j=0}^4 \frac{1}{j!}}
 F^{(j)}(1-\varepsilon) \left[s-(1-\varepsilon)\right]^j,
 \quad &\forall\,s\geq 1-\varepsilon,\\
 F(s),\quad & \forall\, s\in[-1+\varepsilon, 1-\varepsilon],\\
 \displaystyle{\sum_{j=0}^4
 \frac{1}{j!}} F^{(j)}(-1+\varepsilon)\left[ s-(-1+\varepsilon)\right]^j,
 \quad &\forall\, s\leq -1+\varepsilon.
 \end{cases}
 \label{defF1s}
\end{equation}
Under the assumption (A2), it easily follows (see, e.g., \cite{FG2012}) that there exists $\varepsilon^\ast \in (0,\kappa]$ such that, for any $\varepsilon \in (0,\varepsilon^\ast]$, the
function $F_\varepsilon \in \mathcal{C}^4(\mathbb{R})$ and fulfils the following properties
\begin{equation}
\label{Freg}
\alpha_1 s^4 -\gamma\leq F_\varepsilon(s), \quad \alpha_2\leq F_\varepsilon''(s)\leq M,\quad \forall \, s\in \mathbb{R}
\end{equation}
where $\alpha_1, \alpha_2$ and $\gamma$ are positive constants independent of $\varepsilon$, whereas $M$ depends on $\varepsilon$.
Moreover, the regularized potential $\Psi_\varepsilon \in \mathcal{C}^4(\mathbb{R})$ satisfies
\begin{equation}
\label{Freg2}
\Psi_{\varepsilon}(s)\leq \Psi(s), \quad \forall\, s\in [-1,1], \quad |\Psi'_{\varepsilon}(s)| \leq |\Psi'(s)| , \quad \forall\, s\in (-1,1).
\end{equation}
\smallskip

\textbf{Step 2: Approximation of the initial datum.}
We follow here the construction introduced in \cite[Theorem 4.1]{GMT2018}.
For $k\in \mathbb{N}$, we consider the globally Lipschitz function $h_k: \mathbb{R}\rightarrow \mathbb{R}$ such that
\begin{equation}
\label{trunc}
h_k(s)=
\begin{cases}
-k,\quad&s<-k, \\
s,\quad &s\in [-k,k],\\
k,\quad &s>k.
\end{cases}
\end{equation}
Recalling that $\widetilde{\mu}_0=-\Delta \vp_0 +F'(\vp_0)$, we consider $\widetilde{\mu}_{0}^k=h_k \circ \widetilde{\mu}_0$. Since $\widetilde{\mu}_0\in V$ and $h_k$ is Lipschitz, we infer that $\widetilde{\mu}_{0}^k \in V$, for any $k>0$, and we have 
$$
\nabla \widetilde{\mu}_{0}^k=\nabla \widetilde{\mu}_0 \cdot\chi_{[-k,k]} (\widetilde{\mu}_0),\quad \|\widetilde{\mu}_{0}^k\|_V\leq \| \widetilde{\mu}_0\|_V.
$$
In particular, we have
$\| \widetilde{\mu}_{0}^k - \widetilde{\mu}_0\|\rightarrow 0$ as 
$ k \rightarrow \infty$.
Next, for $k\in \mathbb{N}$, we consider the Neumann problem
\begin{equation}
\label{ellapp}
\begin{cases}
-\Delta \vp_{0}^k+F'(\vp_{0}^k)=\widetilde{\mu}_{0}^k,\quad &\text{ in }\Omega,\\
\partial_\n \vp_{0}^k=0, \quad &\text{ on }\partial \Omega.
\end{cases}
\end{equation}
Thanks to Theorem \ref{ell2}, there exists a 
unique solution
$\vp_{0}^k\in H^2(\Omega)$ such that
$F'(\vp_{0}^k) \in H$. The solution $\vp_{0}^k$ satisfies \eqref{ellapp} almost everywhere in $\Omega$ and $\partial_\n \vp_{0}^k=0$ on $\partial \Omega$.
In addition, we have
\begin{equation}
\label{psikH2}
\| \vp_{0}^k\|_{H^2(\Omega)}\leq C(1+\| \widetilde{\mu}_{0}\|).
\end{equation}
According to $\widetilde{\mu}_{0}^k\rightarrow \widetilde{\mu_0}$ in $H$, we observe that $\vp_{0}^k \rightarrow \vp_0$ in $V$. This implies that there exist $\overline{m}\in (0,1)$ (depending only on $m$) and $\overline{k}$ sufficiently large such that 
\begin{equation}
\label{psiH1}
\| \vp_{0}^k\|_V \leq 1+\| \vp_0\|_V,
\quad
|\overline{\vp}_{0}^k|\leq \overline{m}<1,
\quad \forall \, k> \overline{k}.
\end{equation}
Now, applying Theorem \ref{ell2} with 
$f=\widetilde{\mu}_{0}^k$, it yields
$$
\| F'(\vp_{0}^k)\|_{L^{\infty}(\Omega)}
\leq \| \widetilde{\mu}_{0}^k\|_{L^{\infty}(\Omega)}\leq k,
$$
which, in turn, implies that there exists a $\delta=\delta(k)>0$ such that
\begin{equation}
\label{psikinf}
\| \vp_{0}^k\|_{L^\infty(\Omega)}\leq 1-\delta.
\end{equation}
We notice at this point that $F'(\vp_{0}^k)\in V$, and so one can deduce that $\vp_{0}^k\in H^3(\Omega)$. Finally, since  $F(s)=F_\varepsilon(s)$ for all $s\in [-1+\varepsilon,1-\varepsilon]$,
we infer from \eqref{psikinf} that, for $\varepsilon\in (0,\overline{\varepsilon})$, where 
$\overline{\varepsilon}=\min \lbrace \frac12 \delta(k), 
\varepsilon^\ast\rbrace$,  
$$
-\Delta \vp_{0}^k+
F'_\varepsilon(\vp_{0}^k)=\widetilde{\mu}_{0}^k, 
$$
which entails 
\begin{equation}
\label{muk2H1}
\|-\Delta \vp_{0}^k+
F'_\varepsilon(\vp_{0}^k)\|_V\leq \| \widetilde{\mu}_0\|_V.
\end{equation}
\smallskip

\textbf{Step 3: Regularized problem.}
For any $k>\overline{k}$ and $\varepsilon \in (0, \overline{\varepsilon})$,
let us consider the HSCH system with regular potential $\Psi_\varepsilon$ and initial condition $\vp_0^k$.
For simplicity of notation, we will denote the solution by $(\uu_\varepsilon,p_\varepsilon,\vp_\varepsilon)$ keeping in mind the dependence on both $\varepsilon$ and $k$. The system reads as follows
\begin{equation}
\begin{cases}
\label{CHHS-reg}
\nu(\vp_\varepsilon) \uu_\varepsilon +\nabla p_\varepsilon=  \mu_\varepsilon \nabla \vp_\varepsilon,\\
\mathrm{div}\, \uu_\varepsilon=0, \\
\partial_t \vp_\varepsilon+ \uu_\varepsilon \cdot \nabla 
\vp_\varepsilon = \Delta \mu_\varepsilon,\\
\mu_\varepsilon= - \Delta \vp_\varepsilon +  \Psi'_\varepsilon(\vp_\varepsilon)
\end{cases}
 \quad \text{ in } \Omega\times (0,T),
\end{equation}
subject to the boundary and initial conditions
\begin{equation}
\label{bdini-2}
\u \cdot \n= \partial_\n \mu=\partial_\n \vp=0 \quad \text{on\ } \partial \Omega \times (0,T),\quad 
\vp( \cdot,0)=\vp_{0}^k \quad \text{in } \Omega.
\end{equation}
We recall that $\vp_{0}^k\in H^3(\Omega)$ such that $\partial_\n \vp_{0}^k=0$ on $\partial \Omega$ as defined in the previous step. Thanks to \cite[Theorem 1.1]{WW12} and \cite[Theorem 3.1]{WZ13}, there exists a global strong solution to \eqref{CHHS-reg}-\eqref{bdini-2} such that, for any $T>0$,
\begin{align}
&\uu_\varepsilon \in \mathcal{C}([0,T],\H_\sigma \cap \mathbf{V})\cap L^2(0,T;\mathbf{H}^3(\Omega)),
\label{regus}\\
& p_\varepsilon \in \mathcal{C}([0,T],V_0)\cap L^2(0,T;H^4(\Omega)),
\label{regps}\\
&\vp_\varepsilon \in \mathcal{C}([0,T],H^3(\Omega))\cap L^2(0,T;H^5(\Omega))\cap H^1(0,T; V), 
\label{regvps}\\
&\mu_\varepsilon \in \mathcal{C}([0,T],V)\cap L^2(0,T;H^3(\Omega)). \label{regmus}
\end{align}
Let us mention that \cite[Theorem 1.1]{WW12} and \cite[Theorem 3.1]{WZ13} are proven for the HSCH system \eqref{CHHS-reg} with periodic boundary conditions. It is apparent that the proof can be recasted with the above boundary conditions in a smooth bounded domain.
\smallskip

The main part of the proof is now showing global {\it a priori }estimates for the solution which are uniform with respect to the approximating parameters $k$ and $\varepsilon$. In the rest of the proof,  $C$ will denote a generic positive constant, which depends on the parameters of the system, the constants arising from embedding and interpolation results, the norm of $\| \vp_0\|_V$ and $\overline{m}$, but it is independent of $k$, $\varepsilon$ and $\| \widetilde{\mu}_0\|_V$. 
\smallskip

\textbf{Step 4: Energy estimates.} 
Integrating \eqref{CHHS-reg}$_3$ over $\Omega$, and using \eqref{psiH1}, we obtain
\begin{equation}
\label{E1}
|\overline{\vp_\varepsilon}(t)|=\Big|\frac{1}{|\Omega|}\int_{\Omega} \vp_\varepsilon (t)\, \d x \Big|=\Big|\frac{1}{|\Omega|}\int_{\Omega} \vp_0^k\, \d x\Big| \leq \overline{m}.
\end{equation}
We multiply \eqref{CHHS-reg}$_3$ by $\mu$ and \eqref{CHHS-reg}$_4$ by 
$\vp_t$. After integrating over $\Omega$, we get
$$
\ddt \E_\varepsilon(\vp_\varepsilon) + (\uu_\varepsilon\cdot \nabla \vp_\varepsilon, \mu_\varepsilon) +\| \nabla \mu_\varepsilon\|^2=0,
$$
having set
$$
\E_\varepsilon(\vp_\varepsilon)= \int_{\Omega}  \frac12 |\nabla \vp_\varepsilon|^2 +
 \Psi_\varepsilon(\vp_\varepsilon) \, \d x.
 $$
Multiplying \eqref{CHHS-reg}$_1$ by $\uu_\varepsilon$ and adding the resulting equation to \eqref{E1}, we find
$$
\ddt \E_\varepsilon(\vp_\varepsilon) + \int_{\Omega} \nu(\vp_\varepsilon) |\uu_\varepsilon|^2 + | \nabla \mu_\varepsilon|^2 \, \d x=0.
$$
After integrating on the time interval $[0,t]$, we have
\begin{equation}
\label{EI}
\E_\varepsilon(\vp_\varepsilon(t)) +\int_0^t \int_{\Omega} \nu(\vp_\varepsilon) |\uu_\varepsilon|^2 + | \nabla \mu_\varepsilon|^2 \, \d x \, \d \tau=\E_\varepsilon(\vp_0^k).
\end{equation}
By using \eqref{Freg2}, \eqref{psiH1} and \eqref{psikinf}, we notice that 
\begin{equation}
\label{EI2}
\E_\varepsilon(\vp_{0}^k)\leq C(1+\| \vp_{0}\|_V^2).
\end{equation}
According to \eqref{poincare}, \eqref{ipo-nu}, \eqref{Freg}, 
\eqref{E1} and \eqref{EI}, we deduce the bounds
\begin{equation}
\label{E2}
\| \vp_\varepsilon\|_{L^\infty(0,T;V)}\leq C,\quad
\| \nabla \mu_\varepsilon\|_{L^2(0,T;H)}\leq C,\quad
\| \uu_\varepsilon\|_{L^2(0,T;\H_\sigma)}\leq C.
\end{equation}
\smallskip

\textbf{Step 5: Elliptic estimates.}
We multiply \eqref{CHHS-reg}$_4$ by $-\Delta \vp$ and we integrate over $\Omega$. After integrating by parts and using the boundary conditions, we have
$$
\|\Delta \vp_\varepsilon\|^2+ \int_{\Omega} F_\varepsilon''(\vp_\varepsilon)|\nabla \vp_\varepsilon|^2\, \d x= 
(\nabla\mu_\varepsilon, \nabla \vp_\varepsilon)+ \theta_0 \|\nabla\vp_\varepsilon\|^2.
$$
By the regularity theory of the Neumann problem, together with \eqref{Freg2} and \eqref{E2}, we obtain
\begin{equation}
\label{E3}
\| \vp_\varepsilon\|_{H^2(\Omega)}^2\leq C (1+ \| \nabla \mu_\varepsilon\|).
\end{equation}
Due to the monotonicity of $F'_\varepsilon$, it follows that (see, e.g., \cite{FG2012})
$$
\| F'_\varepsilon(\vp_\varepsilon)\|_{L^1(\Omega)}\leq
C \int_{\Omega} \big( F'_\varepsilon(\vp_\varepsilon)-
\overline{F'_\varepsilon(\vp_\varepsilon)} \big) (\vp_\varepsilon-\overline{\vp_\varepsilon}) \, \d x+ C.
$$
Then, multiplying \eqref{CHHS-reg}$_4$ by $\vp_\varepsilon-\overline{\vp_\varepsilon}$ and integrating over $\Omega$, we have
$$
\| \nabla \vp_\varepsilon\|^2+ \int_{\Omega} \big( F'_\varepsilon(\vp_\varepsilon)-
\overline{F'_\varepsilon(\vp_\varepsilon)} \big) (\vp_\varepsilon-\overline{\vp_\varepsilon}) \, \d x \leq 
C (1+ \| \nabla \mu_\varepsilon\|).
$$
Here we have used \eqref{poincare} and \eqref{E2}.
Hence, we deduce from the above estimates that
$$
\| F'_\varepsilon(\vp_\varepsilon)\|_{L^1(\Omega)} \leq C( 1+ \| \nabla \mu_\varepsilon\| ). 
$$
Observing that $\overline{\mu_\varepsilon}=\theta_0 \overline{\vp_\varepsilon}+ \overline{F'_\varepsilon(\vp_\varepsilon)}$, we have
\begin{equation}
\label{E4}
\| \mu_\varepsilon\|_V\leq C(1+\| \nabla \mu_\varepsilon\|).
\end{equation}
Let us now rewrite \eqref{CHHS-reg}$_4$ as 
\begin{equation}
\label{ELL-reg}
-\Delta \vp_\varepsilon+ F'_\varepsilon(\vp_\varepsilon)=f,
\end{equation}
where $f= \mu_\varepsilon+ \theta_0 \vp_\varepsilon$. 
Multiplying \eqref{ELL-reg} by $|F'_\varepsilon(\vp_\varepsilon)|^{p-2}F'_\varepsilon(\vp_\varepsilon)$ and integrating by parts, we find
$$
(p-1)\int_{\Omega} 
|F_\varepsilon'(\vp_\varepsilon)|^{p-2} 
F_\varepsilon''(\vp_\varepsilon)|\nabla \vp_\varepsilon|^2 \, \d x+  \| F'_\varepsilon(\vp_\varepsilon)\|_{L^p(\Omega)}^p =
(f,|F'_\varepsilon(\vp_\varepsilon)|^{p-2}F'_\varepsilon(\vp_\varepsilon)).
$$
Notice that the first term on the left-hand side is positive due to \eqref{Freg2}. By H\"{o}lder inequality, we are led to
$$
\| F'_\varepsilon(\vp_\varepsilon)\|_{L^p(\Omega)}
\leq \| f\|_{L^p(\Omega)}.
$$
By the embedding $V\hookrightarrow L^p(\Omega)$, $1\leq p<\infty$, together with \eqref{E2} and \eqref{E4}, we arrive at
$$
\| F'_\varepsilon(\vp_\varepsilon)\|_{L^p(\Omega)}
\leq C (1+ \|\nabla \mu_\varepsilon\|),
$$
for any $1\leq p<\infty$, where $C$ is a positive constant which depends on $p$.
Writing \eqref{ELL-reg} as $-\Delta \vp_\varepsilon= f-F'_\varepsilon(\vp_\varepsilon)$, we infer from the regularity theory of the Neumann problem  that
\begin{equation}
\label{E5}
\| \vp_\varepsilon\|_{W^{2,p}(\Omega)}\leq C(1+ \|\nabla \mu_\varepsilon\|), 
\end{equation}
for any $p$ as above.
\smallskip

\textbf{Step 6: Time derivative and vorticity estimates.}
We proceed with a control on $\partial_t \vp_\varepsilon$. By using \eqref{BWd2}, \eqref{E2} and \eqref{E3}, 
we deduce that
\begin{align}
\|\partial_t \vp_\varepsilon \|_{V_0'}
&\leq  \|\nabla\mu_\varepsilon\| + \|\uu_\varepsilon\| \| \vp_\varepsilon \|_{L^\infty(\Omega)} \notag \\
&\leq   \| \nabla \mu_\varepsilon\| + C\| \uu_\varepsilon\| \| \vp_\varepsilon\|_V \log^{\frac12} \big( \mathrm{e}+\| \vp_\varepsilon\|_{H^2(\Omega)}\big) \notag \\
&\leq \|\nabla \mu_\varepsilon \|
+C\| \uu_\varepsilon\| \log^\frac12 \big( C +C\| \nabla \mu_\varepsilon\| \big).
\label{E6}
\end{align}
Next, we study the equation for the vorticity derived from the Darcy's law. 
We compute the curl of \eqref{CHHS-reg}$_1$ and we obtain
\begin{equation}
\label{curlu}
\nu(\vp_\varepsilon) \mathrm{curl}\, \uu_\varepsilon
+ \nu'(\vp_\varepsilon)\nabla \vp_\varepsilon \cdot \uu_\varepsilon^{\perp} =
\nabla \mu_\varepsilon\cdot(\nabla \vp_\varepsilon)^\perp,
\end{equation}
where $\vv^\perp= (v_2,-v_1)$.
By using \eqref{BWd2}, \eqref{ipo-nu} and \eqref{E5},
we infer that
\begin{align*}
\nu_\ast \| \mathrm{curl}\, \uu_\varepsilon\|
&\leq C\| \nabla \vp_\varepsilon \cdot \uu_\varepsilon^\perp\|+ \| \nabla \mu_\varepsilon \cdot (\nabla \vp_\varepsilon)^\perp\|\\
&\leq C\| \uu_\varepsilon\| \| \nabla \vp_\varepsilon\|_{L^\infty(\Omega)}+
C \| \nabla \mu_\varepsilon \| \| \nabla \vp_\varepsilon\|_{L^\infty(\Omega)}\\
&\leq C(\| \uu_\varepsilon\|+\| \nabla \mu_\varepsilon\|)\| \vp_\varepsilon\|_{H^2(\Omega)} \log^{\frac12}(e+\| \vp_\varepsilon\|_{W^{2,3}(\Omega))})\\
&\leq C (\| \uu_\varepsilon\|+\| \nabla \mu_\varepsilon\|) (1+\| \nabla\mu_\varepsilon\|)^\frac12 \log^{\frac12}(C+C\| \nabla \mu_\varepsilon\|).
\end{align*}
Hence, by \eqref{rot} we are led to
\begin{equation}
\label{E7}
\| \uu_\varepsilon\|_{V}\leq
C (\| \uu_\varepsilon\|+\| \nabla \mu_\varepsilon\|)
(1+\| \nabla\mu_\varepsilon\|)^\frac12 \log^{\frac12}(C+C\| \nabla \mu_\varepsilon\|).
\end{equation}
\smallskip

\textbf{Step 7: Higher order differential equality.}
Let us introduce the notation $\partial_t^h v(\cdot)=\frac{1}{h} (v(\cdot+h)-v(\cdot))$. By the regularity \eqref{regvps} and the control $\| \partial_t^h v\|_{L^2(0,T;X)}\leq \| \partial_t v\|_{L^2(0,T+1;X)}$, is is easily seen that $\| \partial_t^h \mu_\varepsilon\|_{L^2(0,T; V')}\leq C$, where $C$ is independent of $h$. This entails that $ \partial_t \mu_\varepsilon \in L^2(0,T;V')$ and one can write
$$
\l \partial_t \mu_\varepsilon, v\r = (\nabla \partial_t \vp_\varepsilon, \nabla v ) + (\Psi''_\varepsilon(\vp_\varepsilon) \partial_t \vp_\varepsilon,v), \quad \forall \, v\in V,
$$
for almost every $t\in (0,T)$.
Taking the duality between $\partial_t \mu_\varepsilon$ and \eqref{CHHS-reg}$_3$, and using the expression above, we obtain
\begin{equation}
\label{eq1}
 \frac12 \ddt  \| \nabla \mu_\varepsilon\|^2
 +  \| \nabla \partial_t\vp_\varepsilon\|^2+
  \int_\Omega F_\varepsilon''(\vp_\varepsilon)|\partial_t \vp_\varepsilon|^2 \, \d x =
  -\l \u_\varepsilon \cdot \nabla \vp_\varepsilon, \partial_t \mu_\varepsilon \r + \theta_0 \|\partial_t \vp_\varepsilon\|^2,
\end{equation}
for almost every $t\in [0,T]$.
We now want to differentiate \eqref{CHHS-reg}$_1$ with respect to time and multiplying the resulting equation by $\uu_\varepsilon$. However, this is only a formal procedure at this stage since it is not known the regularity of $\partial_t \uu_\varepsilon$ due to the  non-constant viscosity. Nonetheless, we can prove the equality
\begin{equation}
\label{eq2}
\frac12  \ddt \int_{\Omega} \nu(\vp_\varepsilon)|\uu_\varepsilon|^2\, \d x = 
\l \partial_t \mu_\varepsilon,\nabla \vp_\varepsilon\cdot \uu_\varepsilon \r +  \int_\Omega \mu_\varepsilon\nabla \partial_t \vp_\varepsilon \cdot \uu_\varepsilon \, \d x- \frac12 \int_{\Omega} \nu'(\vp_\varepsilon) \partial_t \vp_\varepsilon |\uu_\varepsilon|^2 \, \d x, 
\end{equation}
for almost every $t\in [0,T]$.
To show this, we consider \eqref{CHHS-reg}$_1$ evaluated at $t+h$, for $h>0$, and at $t$ and we multiply them by $\frac12 (\uu_\varepsilon(t+h)-\uu_\varepsilon (t))$. Adding up the two expressions and integrating over $\Omega$, we get
\begin{align}
\frac12 \int_{\Omega}  \nu(\vp_\varepsilon(t+h))&\uu^2_\varepsilon(t+h)-\nu(\vp_\varepsilon(t))\uu^2_\varepsilon(t) \,\d x -\int_{\Omega} \frac12 \Big( \nu(\vp_\varepsilon(t+h))-\nu(\vp_\varepsilon(t))\Big) \uu_\varepsilon(t)\uu_\varepsilon(t+h) \, \d x \notag\\
&=\frac12 \int_{\Omega} \Big( \mu_\varepsilon(t+h)\nabla \vp_\varepsilon(t+h) +  \mu_\varepsilon(t)\nabla \vp_\varepsilon(t)\Big)(\uu_\varepsilon(t+h)-\uu_\varepsilon (t))
\, \d x.
\label{Testh1}
\end{align}
Next, we multiply \eqref{CHHS-reg}$_1$ evaluated at $t+h$ and at $t$ by $\uu_\varepsilon(t)$ and $ (\uu_\varepsilon(t+h)$, respectively. Adding up and integrating over $\Omega$, we find
\begin{align}
\int_{\Omega} \nu(\vp_\varepsilon(t+h))& \uu_\varepsilon(t+h)\uu_\varepsilon(t) - \nu(\vp_\varepsilon(t)) \uu_\varepsilon(t+h)\uu_\varepsilon(t) \, \d x\notag \\
&=\int_{\Omega} \mu(\vp_\varepsilon(t+h))\nabla \vp_\varepsilon(t+h)\uu_\varepsilon(t)-
\mu_\varepsilon(t) \nabla \vp_\varepsilon(t)\uu_\varepsilon(t+h) \,\d x. \label{Testh2} 
\end{align}
Adding the two expressions \eqref{Testh1} and \eqref{Testh2}, and multiplying the resulting equality by $\frac1h$, we eventually obtain 
\begin{align}
\partial_t^h &\int_\Omega \frac{\nu(\vp_\varepsilon)}{2} |\uu_\varepsilon|^2 \, \d x= \frac12 \int_\Omega \partial_t^h \mu_\varepsilon \nabla \vp_\varepsilon(t+h) \cdot \big( \uu_\varepsilon(t+h)+\uu_\varepsilon(t)
\big) \,\d x \notag \\
& + \frac12 \int_{\Omega} \mu_\varepsilon(t)\partial_t^h \nabla \vp_\varepsilon \cdot \big( \uu_\varepsilon(t+h)+\uu_\varepsilon(t)
\big) \,\d x- \frac12 \int_{\Omega} \partial_t^h \nu(\vp_\varepsilon) \uu_\varepsilon(t+h)\cdot \uu_\varepsilon(t)\, \d x,
\label{Testh3}
\end{align}
for almost every $t\in (0,T)$.
Taking $\psi \in \mathcal{C}_0^\infty(0,T)$, by using the definition of weak derivative, integration by parts and \eqref{Testh3}, we have
\begin{align*}
\int_{0}^T &\ddt \int_{\Omega} \frac{\nu(\vp_\varepsilon)}{2}|\uu_\varepsilon|^2 \, \d x \, \psi \, \d \tau= - \int_{0}^T \int_{\Omega} \frac{\nu(\vp_\varepsilon)}{2}|\uu_\varepsilon|^2 \, \d x \, \ddt\psi \, \d \tau\\
&= \lim_{h\rightarrow 0} - \int_{0}^T \int_{\Omega} \frac{\nu(\vp_\varepsilon)}{2}|\uu_\varepsilon|^2 \, \d x \, \frac{\psi(\tau)-\psi(\tau-h)}{h} \, \d \tau 
= \lim_{h\rightarrow 0} \int_0^T \partial_t^h \int_\Omega \frac{\nu(\vp_\varepsilon)}{2} |\uu_\varepsilon|^2 \, \d x \, \psi \, \d \tau\\
& = \lim_{h\rightarrow 0}\,  \frac12 \int_0^T \big\l \partial_t^h \mu_\varepsilon, \nabla \vp_\varepsilon(\tau+h) \cdot  \big( \uu_\varepsilon(\tau+h)+\uu_\varepsilon(\tau)
\big) \big\r \psi \, \d \tau \notag \\
&\quad + \lim_{h\rightarrow 0}\,  \frac12 \int_0^T \int_{\Omega} \mu_\varepsilon(\tau)\partial_t^h \nabla \vp_\varepsilon \cdot \big( \uu_\varepsilon(\tau+h)+\uu_\varepsilon(\tau) \big) \,\d x \,  \psi \, \d \tau\\
&\quad - \lim_{h\rightarrow 0}\,  \frac12 \int_0^T\int_{\Omega}  \partial_t^h \nu(\vp_\varepsilon) \uu_\varepsilon(\tau+h)\cdot \uu_\varepsilon(\tau)\, \d x \, \psi \, \d \tau.
\end{align*}
Exploiting the regularity properties \eqref{regus}-\eqref{regmus}, and recalling that $\partial_t^h u \rightarrow \partial_t u$ in $L^2(0,T;X)$, we can pass to the limit as $h\rightarrow 0$ and we finally deduce \eqref{eq2}. 

Now, summing up \eqref{eq1} and \eqref{eq2}, we arrive at the differential equality
\begin{align}
\ddt  H_\varepsilon
 &+  \| \nabla \partial_t\vp_\varepsilon\|^2+
  \int_\Omega F_\varepsilon''(\vp_\varepsilon)|\partial_t \vp_\varepsilon|^2 \, \d x \notag \\   
&=\theta_0 \|\partial_t \vp_\varepsilon\|^2+
  \int_\Omega \mu_\varepsilon\nabla \partial_t \vp_\varepsilon \cdot \uu_\varepsilon \, \d x - \frac12 \int_{\Omega} \nu'(\vp_\varepsilon) \partial_t \vp_\varepsilon |\uu_\varepsilon|^2 \, \d x, 
 \label{Lambda-eq}
\end{align}
for almost every $t\in (0,T)$, where
$$
H(t)= \frac12 \| \nabla \mu_\varepsilon(t)\|^2 +  \frac12  \int_{\Omega} \nu(\vp_\varepsilon(t))|\uu_\varepsilon(t)|^2\, \d x.
$$
Notice that it is essential the cancellation of the troublesome term $\l \partial_t \mu_\varepsilon,\nabla \vp_\varepsilon\cdot \uu_\varepsilon \r$ on the right-hand side of \eqref{eq1} and \eqref{eq2}.
\smallskip

\textbf{Step 8: Higher order estimates.}
First, by assumption (A1), there exists a constant $C$ such that
\begin{equation}
\label{Lambda}
\frac{1}{C}\big( \| \nabla \mu_\varepsilon\|^2+ \| \uu_\varepsilon\|^2 \big)\leq  H \leq C \big( \| \nabla \mu_\varepsilon\|^2+ \| \uu_\varepsilon\|^2 \big).
\end{equation}
We proceed by estimating the three remaining terms on the right-hand side of \eqref{Lambda-eq}. In doing so we will make use of the following estimates that follow from \eqref{E6}, \eqref{E7} and \eqref{Lambda}
\begin{equation}
\label{E6b}
\| \partial_t \vp_\varepsilon\|_{V_0'}
\leq C H^\frac12 \log^\frac12 (C+C H),
\end{equation}
and
\begin{equation}
\label{E7b}
\| \uu_\varepsilon\|_V \leq C (H^\frac12 +H^\frac34) \log^\frac12(C+CH). 
\end{equation}
Since $\overline{\partial_t \vp_\varepsilon}=0$, by using \eqref{I}, \eqref{E6b} and Young's inequality, the first term on right-hand side of \eqref{Lambda-eq} is simply controlled by
\begin{align}
\theta_0 \| \partial_t \vp_\varepsilon\|^2\leq C \| \partial_t \vp_\varepsilon\|_{V_0'} \|\nabla \partial_t \vp_\varepsilon\| 
\leq \frac14 \|\nabla \partial_t \vp_\varepsilon\|^2
+ C H \log \big( C+CH).
\label{1term}
\end{align}
Regarding the second term in \eqref{Lambda-eq}, by exploiting \eqref{L}, \eqref{I}, the regularity of the Neumann problem, the equation \eqref{CHHS-reg}$_3$, and the estimates \eqref{E4} and \eqref{Lambda}, we obtain
\begin{align}
\int_\Omega \mu_\varepsilon \nabla \partial_t \vp_\varepsilon \cdot \uu_\varepsilon \, \d x
&=-\int_\Omega \partial_t \vp_\varepsilon \uu_\varepsilon \cdot\nabla \mu_\varepsilon \, \d x\nonumber\\
&\leq \| \uu_\varepsilon\| 
\| \partial_t \vp_\varepsilon\|_{L^4(\Omega)}
\| \nabla \mu_\varepsilon\|_{L^4(\Omega)}
\nonumber\\
& \leq C \|\u_\varepsilon\| 
\| \partial_t \vp_\varepsilon\|^{\frac12} \|\nabla \partial_t \vp_\varepsilon\|^\frac12
\|\nabla \mu_\varepsilon\|^{\frac12} 
\| \mu_\varepsilon\|_{H^2(\Omega)}^\frac12 \nonumber \\
&\leq C \|\u_\varepsilon \| 
\| \partial_t \vp_\varepsilon\|_{V_0'}^{\frac14} 
\|\nabla \partial_t \vp_\varepsilon\|^\frac34
\|\nabla \mu_\varepsilon\|^{\frac12} 
\big( \| \mu_\varepsilon\| + \| \Delta \mu_\varepsilon\| \big)^\frac12
 \nonumber\\
&\leq C \| \u_\varepsilon\| 
\| \partial_t \vp_\varepsilon\|_{V_0'}^\frac14
\| \nabla \partial_t \vp_\varepsilon\|^\frac34 
\|\nabla \mu_\varepsilon\|^{\frac12}
\big( 1+ \| \nabla \mu_\varepsilon\|+\|\partial_t \vp_\varepsilon\|+ 
\| \u_\varepsilon\cdot \nabla \vp_\varepsilon\| \big)^\frac12 \nonumber\\
&=W_1+W_2+W_3+W_4.
\label{WWW}
\end{align}
By \eqref{Lambda} and \eqref{E6b}, we have
\begin{align}
W_1 & \leq \frac{1}{32} \| \nabla \partial_t \vp_\varepsilon\|^2
+ C \| \uu_\varepsilon\|^\frac85 \| \partial_t \vp_\varepsilon\|_{V_0'}^{\frac25} \| \nabla \mu_\varepsilon\|^\frac45 \notag\\
&\leq \frac{1}{32} \| \nabla \partial_t \vp_\varepsilon\|^2+ C 
H^{\frac75} \log^{\frac15} (C+C H),
\label{W1}
\end{align}
and
\begin{align}
W_2&\leq 
\frac{1}{32} \| \nabla \partial_t \vp_\varepsilon\|^2
+ C \| \uu_\varepsilon\|^\frac85 \| \partial_t \vp_\varepsilon\|_{V_0'}^{\frac25} \| \nabla \mu_\varepsilon\|^\frac85 \notag \\
&\leq \frac{1}{32} \| \nabla \partial_t \vp_\varepsilon\|^2+
C H^\frac95 \log^{\frac15} (C+CH).
\label{W2}
\end{align}
By using \eqref{I}, \eqref{Lambda} and \eqref{E6b}, we deduce that
\begin{align}
W_3 &\leq   C \| \u_\varepsilon\|  \| \partial_t\vp_\varepsilon\|_{V_0'}^\frac12
\| \nabla \partial_t \vp_\varepsilon\|
\|\nabla \mu_\varepsilon\|^{\frac12} \notag\\
&\leq  \frac{1}{32} \| \nabla \partial_t \vp_\varepsilon\|^2+
C \| \u_\varepsilon\|^2   \| \partial_t\vp_\varepsilon\|_{V_0'} \| \nabla \mu_\varepsilon\| \notag \\
&\leq  \frac{1}{32} \| \nabla \partial_t \vp_\varepsilon\|^2+
C H^2 \log^\frac12 (C+CH).
\label{W3}
\end{align}
Thanks to \eqref{BWd2}, \eqref{E3}, \eqref{E5}, \eqref{Lambda} and \eqref{E6b}, it follows that
\begin{align}
W_4&\leq C \| \u_\varepsilon\|^\frac32 
\| \partial_t \vp_\varepsilon\|_{V_0'}^\frac14
\| \nabla \partial_t \vp_\varepsilon\|^\frac34 
\|\nabla \mu_\varepsilon\|^{\frac12}
\| \nabla \vp_\varepsilon\|_{L^\infty(\Omega)}^\frac12 \notag \\
&\leq \frac{1}{32} \| \nabla \partial_t \vp_\varepsilon\|^2 + C
\| \uu_\varepsilon\|^\frac{12}{5} \| \partial_t \vp_\varepsilon\|_{V_0'}^\frac25 \| \nabla \mu_\varepsilon\|^\frac45
\| \vp_\varepsilon\|_{H^2(\Omega)}^\frac45 
\log^{\frac45}(e+\| \vp_\varepsilon\|_{W^{2,3}(\Omega)}) \notag \\
&\leq \frac{1}{32} \| \nabla \partial_t \vp_\varepsilon\|^2 + C
\| \uu_\varepsilon\|^\frac{12}{5} \| \partial_t \vp_\varepsilon\|_{V_0'}^\frac25 \| \nabla \mu_\varepsilon\|^\frac45(1+\| \nabla \mu_\varepsilon\|)^\frac25
\log^{\frac45}(e+\| \vp_\varepsilon\|_{W^{2,3}(\Omega)}) \notag\\
&\leq \frac{1}{32} \| \nabla \partial_t \vp_\varepsilon\|^2
+C(1+H^2) \log(C+C H).
\label{W4}
\end{align}
Thus, combining \eqref{WWW} with \eqref{W1}, \eqref{W2},  \eqref{W3} and  \eqref{W4}, we are led to
\begin{equation}
\label{2term}
\int_\Omega \mu_\varepsilon \nabla \partial_t \vp_\varepsilon \cdot \uu_\varepsilon \, \d x 
\leq \frac14 \| \nabla \partial_t \vp_\varepsilon\|^2 +C (1+H^2) \log(C+CH).
\end{equation}
Let us now control the last term in \eqref{Lambda-eq}. By \eqref{L}, \eqref{I}, \eqref{Lambda}, \eqref{E6b} and \eqref{E7b}, we have
\begin{align}
-\int_\Omega \nu'(\vp_\varepsilon)\partial_t \vp_\varepsilon |\uu_\varepsilon|^2 \, \d x &\leq 
C \| \partial_t \vp_\varepsilon \| 
\|\uu_\varepsilon \|_{L^4(\Omega)}^2 \notag \\
&\leq C \|\partial_t \vp_\varepsilon \|_{V_0'}^\frac12 
\| \nabla \partial_t \vp_\varepsilon \|^\frac12
\| \uu_\varepsilon\| \| \uu_\varepsilon\|_{V}\notag\\
&\leq \frac14 \|\nabla \partial_t \vp_\varepsilon \|^2+
C \| \partial_t \vp_\varepsilon\|_{V_0'}^\frac23 \| \uu_\varepsilon\|^\frac43 \| \uu_\varepsilon\|_V^\frac43 \notag\\
&\leq \frac14 \|\nabla \partial_t \vp_\varepsilon \|^2+
C (1+H^2)\log(C+CH)
\label{3term}
\end{align}
Combining \eqref{Lambda-eq} with \eqref{1term}, \eqref{2term} and \eqref{3term}, we eventually end up with the differential inequality
\begin{equation}
\label{Linequality}
\ddt H + \frac12 
\| \nabla \partial_t \vp_\varepsilon\|^2 \leq  C+ CH^2 \log(\mathrm{e}+H).
\end{equation}
We now observe that, according to \eqref{EI} and \eqref{EI2},  $H \in L^1(0,T)$ for any $T>0$ with $\int_{0}^T H(\tau) \, \d \tau\leq C(1+\| \vp_0\|_V^2)$. Thanks to the continuity in time of the solution (cf. \eqref{regus}, \eqref{regvps} and \eqref{regmus}) and using the relation $\mu_\varepsilon \nabla\vp_\varepsilon=\nabla(\mu_\varepsilon \vp_\varepsilon)- \mu_\varepsilon \nabla \vp_\varepsilon$, we infer from the equation \eqref{CHHS-reg}$_1$ and the estimate \eqref{muk2H1} that 
\begin{align*}
\| \sqrt{\nu(\vp_\varepsilon(0))}\uu_\varepsilon(0)\|
&\leq C\| \mu_\varepsilon(0)\|_V 
\|\vp_\varepsilon(0) \|_{L^\infty(\Omega)}\\
&\leq C \| -\Delta \vp_0^k+F_\varepsilon'(\vp_0^k)-\theta_0\vp_0^k\|_V \\
&\leq C\big( 1+ \| \widetilde{\mu}_0\|_V+ \|\vp_0 \|_V).
\end{align*}
This, in turn, implies 
\begin{equation}
\label{lambda0}
H(0)\leq C\big( 1+ \| \widetilde{\mu}_0\|_V+ \|\vp_0 \|_V \big)^2.
\end{equation}
Therefore, for any $T>0$, the generalized Gronwall lemma \ref{GL2} yields
\begin{equation}
\label{Reg1}
\| \nabla \mu_\varepsilon(t)\|^2+ \| \uu_\varepsilon(t)\|^2\leq \big(\mathrm{e}+C(1+ \| \widetilde{\mu}_0\|_V+ \|\vp_0 \|_V)^2\big)^{\mathrm{e}^{C(1+\| \vp_0\|_V^2)}} \times 
\mathrm{e^{CT \mathrm{e}^{C(1+\| \vp_0\|_V^2)}}}=C_1, 
\end{equation}
for any $t\in[0,T]$, and
\begin{equation}  
\int_0^T  \| \nabla \partial_t \vp_\varepsilon(\tau )\|^2 \, \d \tau \leq C \big( \| \widetilde{\mu}_0\|_V+ \|\vp_0 \|_V \big)^2 +C T +C C_1^2\log(\mathrm{e}+ C_1) T=C_2.
\end{equation}
Here, $C$ is independent of $\varepsilon$ and $k$ as above.
Consequently, we deduce from \eqref{E5}, \eqref{E6} and \eqref{E7} that
\begin{align}
\label{Reg2}
\| \uu_\varepsilon(t)\|_{V}+
\| \vp_\varepsilon(t)\|_{W^{2,p}(\Omega)}+\| \partial_t\vp_\varepsilon(t)\|_{V_0'}+
\| \Psi_\varepsilon'(\vp_\varepsilon)(t)\|_{L^p(\Omega)}\leq C_3, \quad \forall \,
t\in [0,T],
\end{align}
where $C_3$ only depends on $C_1$ and $p$ (for any $2\leq p<\infty$). Applying Theorem \ref{NP-nc}, together with the above estimates \eqref{Reg1} and \eqref{Reg2}, we obtain
\begin{equation}
\label{Reg3}
\| p_\varepsilon(t)\|_{H^2(\Omega)} \leq C_4, \quad \forall \,
t\in [0,T].
\end{equation}
In addition, by using the equation \eqref{CHHS-reg}$_3$ and the regularity of the Neumann problem, we infer from \eqref{Reg2} that
\begin{equation}
\label{Reg4}
\int_{0}^T \| \mu_\varepsilon (\tau)\|_{H^3(\Omega)}^2 \, \d \tau \leq C_5.
\end{equation}
The two constants $C_4$ and $C_5$ depends on $C_1$, $C_2$ and $C_3$, but are independent of $\varepsilon$ and $k$.
\smallskip

\textbf{Step 9: Passage to the limit and regularity.} 
Thanks to the uniform estimates \eqref{Reg1}-\eqref{Reg4} with respect to $k$ and $\varepsilon$, the existence of a global strong solution to \eqref{CHHS}-\eqref{bdini} is recovered by a standard procedure. The solution $(\uu,p,\vp)$ is the limit of the solutions $(\uu_\varepsilon,p_\varepsilon,\vp_\varepsilon)$ by letting $\varepsilon \rightarrow 0$ (with $k$ fixed) and, then, $k \rightarrow\infty$. It is easily seen that it satisfies the regularity properties stated in Theorem \ref{strong2d}. In particular, $|\vp(x,t)|\leq 1$ for almost every $(x,t)\in \Omega\times(0,T)$ (see, e.g., \cite[Section 3.3]{GGW2018}). Furthermore, by exploiting assumption (A3) and  the regularity $\mu\in L^\infty(0,T;V)$, an application of Theorem \ref{ell2} entails $F''(\vp)\in L^\infty(0,T;L^p(\Omega))$, for any $p\in [2,\infty)$. Repeating the same argument used in Step $7$, we infer that $\partial_t \mu\in L^2(0,T;V')$.
\smallskip

\textbf{Step 10: Uniqueness and continuous dependence estimate.} The uniqueness of strong solutions follows from Theorem \ref{weak-strong}. Let us prove a continuous dependence estimate in $L^2(\Omega)$ with respect to the initial data. We consider two initial conditions $\vp_{01}$ and $\vp_{02}$ satisfying the assumptions of Theorem \ref{strong2d} with $\overline{\vp}_{01}=\overline{\vp}_{02}=m \in (-1,1)$. 
We define $\uu=\uu_1-\uu_2$, $p=p_1-p_2$ and $\vp=\vp_1-\vp_2$. 
The problem for the difference of two solutions read as
$$
\partial_t \vp+ \u_1 \cdot \nabla \vp+ \u\cdot \nabla \vp_2 =   \Delta \mu,
$$
where 
$$
\nu(\vp_1) \u +\nabla \widetilde{p}= -\Delta \vp \nabla \vp_2-\Delta \vp_1 \nabla \vp- (\nu(\vp_1)-\nu(\vp_2)) \uu_2,\quad 
\mu= - \Delta \vp +  \Psi'(\vp_1)-\Psi'(\vp_2).
$$
Notice that $\overline{\vp}(t)=0$ for all $t\geq 0$. By the regularity properties of strong solutions, we have the estimates
\begin{equation}
\label{Reg-diff}
\| \u_i\|_{L^\infty(0,T;V)}\leq C,
\quad \|\vp_i\|_{L^\infty(0,T;W^{2,p}(\Omega))}\leq C, 
\quad \|F''(\vp_i)\|_{L^\infty(0,T;L^p(\Omega))}\leq C,
\end{equation}
for any $2\leq p<\infty$.
Following the proof of Theorem \ref{RU}, it is easily seen that
$$
\frac12 \ddt \| \vp\|^2 + \frac12 \| \Delta \vp\|^2
\leq C \|\vp \|^2 + C \| \uu\| \| \nabla \vp\|.
$$
Now we multiply the Darcy's law above by $\uu$ and we integrate over $\Omega$. By using the Sobolev embedding and \eqref{Reg-diff}, we obtain
\begin{align*}
\nu_\ast \| \u\|^2
&\leq C\|\Delta \vp \| \|\nabla \vp_2 \|_{L^\infty(\Omega)}
\| \u\|+ C \| \Delta \vp_1\|_{L^4(\Omega)} \|\nabla \vp\|_{L^4(\Omega)} \| \u\|+ C\| \vp\|_{L^6(\Omega)} \| \u_2\|_{L^3(\Omega)} \|\u \|\\
&\leq C \| \Delta \vp\| \| \uu\|.
\end{align*}
Recalling the first inequality in \eqref{H1-inter}, we eventually deduce the differential inequality 
\begin{equation}
\label{diff-ho}
\frac12 \ddt \| \vp\|^2 + \frac14 \| \Delta \vp\|^2
\leq C \|\vp \|^2,
\end{equation}
which implies the desired conclusion \eqref{CDstrong}.
The proof is complete.
\end{proof}

We conclude this section by showing the propagation in time of regularity for any weak solution.

\begin{theorem}
\label{regularity}
Let $d=2$ and the initial datum $\vp_0$ be such that $\|\vp_0)\|_V \leq R$, for some $R>0$, $\| \vp_0\|_{L^\infty(\Omega)}\leq 1$ and $\overline{\vp}_0=m\in (-1,1)$. Assume that (A1)-(A3) hold. 
Then, for every $\sigma>0$ and $p\geq1$, there exists a constant
$C=C(\sigma,p,m,R)>0$ such that
\begin{equation}
\label{weak-reg-sigma}
 \| \u\|_{L^{\infty}(\sigma,\infty;V)}\leq C,\quad 
 \| \vp\|_{L^{\infty}(\sigma,\infty;W^{2,p}(\Omega))}
  \leq C.
\end{equation}
Moreover, for every $\sigma>0$, there exist  $\delta=\delta(\sigma, m,R)\in (0,1)$  and a constant $C=C(\sigma, m, R)>0$ such that
\begin{equation}
\label{seppropest}
\| \vp(t)\|_{L^{\infty}(\Omega)}\leq 1-\delta,
\quad \forall \, t \geq 2\sigma,
\end{equation}
and
\begin{equation}
\label{uPvp}
\| \u\|_{L^{\infty}(2\sigma,\infty; H^2(\Omega))}\leq C,\quad 
\| \vp\|_{L^{\infty}(2\sigma,\infty; H^4(\Omega))}
\leq C.
\end{equation}
\end{theorem}

\begin{proof}
Let $(\uu,p,\vp)$ a weak solution to system \eqref{CHHS}-\eqref{bdini} with initial datum $\vp_0$. In light of the energy equality \eqref{BEL}, for any $\sigma>0$, we have
$$
\E(\vp(t))\leq \frac12 R^2+C, \quad \forall\, t \in \Big[0,\frac{\sigma}{2}\Big], \ \text{ and } \ \int_0^{\frac{\sigma}{2}} \| \nabla \mu(\tau)\|^2 \, \d \tau\leq \frac12 R^2+C.
$$
We deduce that there exists $t^\ast \in (0,\frac{\sigma}{2})$ such that $\|\nabla \mu(t^\ast)\|\leq \sqrt{\frac{R^2+2 C}{\sigma}}$ and $\vp(t^\ast)\in H^2(\Omega)$. 
By Theorem \eqref{strong2d}, there exists a unique global strong solution $(\uu^\ast,p^\ast,\vp^\ast)$ on $[t^\ast,\infty)$ with initial datum $\vp(t^\ast)$. In addition, we infer from Theorem \ref{weak-strong} that $(\uu,p,\vp)\equiv (\uu^\ast,p^\ast,\vp^\ast)$ on $[t^\ast,\infty)$. Now, by repeating the higher order estimates in the proof of Theorem \eqref{strong2d} for the solution $(\uu^\ast,p^\ast,\vp^\ast)$, we arrive at
$$
\ddt H + \frac12 \| \nabla \partial_t \vp^\ast\|^2\leq C +CH^2 \log(\mathrm{e}+H),
$$
where $H= \frac12 \| \nabla \mu^\ast\|^2+\frac12 \|\sqrt{\nu(\vp^\ast)}\uu^\ast \|^2$.
Because of the energy equality \eqref{BEL}, $\int_{t}^{t+r}H(\tau) \, \d \tau \leq C(R^2+1)$, for any $t\geq 0$. It follows from the generalized uniform Gronwall lemma \ref{UGL2} that 
$$
H(t)\leq \mathrm{e}^{\big(\frac{C(R^2+1)}{\sigma}+C\sigma\big)\mathrm{e}^{C(R^2+1)}}, \quad \forall \, t \geq \sigma.
$$
Thanks to the classical inequality $\| \mu^\ast\|_{V}\leq C(1+\| \nabla \mu^\ast\|)$ (cf. Step 5 in the proof of Theorem \ref{strong2d}), we infer that $\| \mu^\ast\|_{L^\infty(\sigma,\infty;V)}\leq C$, where the constant $C>0$ only depends on $R$, $\sigma$ and $m$. As a consequence, combining this estimate with Theorem \ref{ell2} and the equation for the vorticity (cf. \eqref{curlu}), we easily obtain the bounds in \eqref{weak-reg-sigma}.

The proof of \eqref{seppropest} and \eqref{uPvp} relies on further higher-order estimates similar to \cite[Lemma 5.4]{GGW2018} (see also \cite{CG,GMT2018}). We first observe that, repeating the argument in Step 10 and exploiting the estimates in \eqref{weak-reg-sigma}, we find 
$$
\frac12 \ddt \| \partial_t^h \vp^\ast\|^2 + \frac14 \| \Delta \partial_t^h \vp^\ast\|^2
\leq C \|\partial_t^h \vp^\ast \|^2,
$$
for almost everywhere $t\in (\sigma,\infty)$, where $\partial_t^h \vp(\cdot)= \frac{1}{h}(\vp^\ast(\cdot+h)-\vp(\cdot))$, and the constant $C>0$ depends on $\sigma$, $R$ and $m$, but is independent of $h$. 
By the uniform Gronwall lemma \cite[Chapter III, Lemma 1.1]{Temam1997}, after taking the limit as $h\rightarrow 0$, we obtain that $\| \partial_t \vp^\ast\|_{L^\infty(2\sigma,\infty;H)}\leq C$. From the equation \eqref{CHHS}$_3$ and the above regularity, we deduce that $\| \mu^\ast\|_{L^\infty(2\sigma,\infty;H^2(\Omega))}$. Then, Theorem \ref{ell2} immediately entails the validity of \eqref{seppropest}. Finally, the estimates in \eqref{uPvp} can be easily inferred from the separation property and the regularity of the Neumann problem. The proof is complete. 
\end{proof}

\section{Local Strong Solutions in Three Dimensions}
\label{S6}
\setcounter{equation}{0}


The purpose of this section is to show the existence and uniqueness of strong solutions for the HSCH system with unmatched viscosities and logarithmic potential in dimension three. More precisely, we prove that the strong solutions are local-in-time for large initial conditions and global-in-time for appropriate small initial conditions.

Before proceeding with the main result of this section, by virtue of the assumption (A4), we introduce $\beta>0$ such that the potential
$$
\widetilde{\Psi}(s)= F(s)-\frac{\theta_0}{2}s^2+|\Psi(\beta)|,
 \quad s \in [-1,1],
$$
is non-negative. Then, we define the related free energy
$$
\widetilde{\E}(\vp)= \int_{\Omega} \frac12 | \nabla \vp|^2 + \widetilde{\Psi}(\vp) \, \d x.
$$
We are now ready to state 
\begin{theorem}
\label{strong3d}
Let $d=3$. Assume that (A1)-(A4) hold. Let $\vp_0 \in H^2(\Omega)$ such that $\| \vp_0\|_{L^\infty(\Omega)}\leq 1$, $\overline{\vp}_0=m \in (-1,1)$, $\mu_0=-\Delta \vp_0+\Psi'(\vp_0) \in V$ and $\partial_\n \vp_0=0$ on $\partial \Omega$. We have the following:
\smallskip

\begin{itemize}
\item[$\bullet$]
For any $R_1>0$ and $R_2>0$ such that  $\| \vp_0\|_V\leq R_1$ and $\| \widetilde{\mu}_0\|_V\leq R_2$, where $\widetilde{\mu}_0=-\Delta \vp_0+F'(\vp_0)$, there exist $T_0=T_0(R_1,R_2)>0$ and a unique strong solution to \eqref{CHHS}-\eqref{bdini} defined on $[0,T_0]$ such that 
\begin{align}
\label{LT1}
&\uu \in L^\infty(0,T_0;\H_\sigma\cap \V), \quad p\in L^\infty(0,T_0;H^2(\Omega)),\\
\label{LT2}
&\vp \in L^\infty(0,T_0;W^{2,p}(\Omega))\cap H^1(0,T_0;V),\\
\label{Linf3d}
&\vp \in L^{\infty}(\Omega \times (0,T_0)) \text{ with } |\vp(x,t)|<1 \text{ a.e. }(x,t) \in \Omega\times(0,T_0),\\
\label{LT3}
&\mu \in L^{\infty}(0,T_0;V)\cap L^2(0,T_0;H^3(\Omega)),
\end{align}
where $2\leq p\leq 6$, and satisfies \eqref{CHHS} almost everywhere in  $\Omega \times (0,T_0)$ and \eqref{bdini} almost everywhere in $\partial \Omega \times (0,T_0)$, and $\vp(\cdot,0)=\vp_0(\cdot)$.

\item[$\bullet$]
There exist two constants $\eta_1>0$ and $\eta_2>0$ depending on the parameters of the system. If the initial condition $\vp_0$ satisfies  
\begin{equation}
\label{assum3}
\widetilde{\E}(\vp_0)\leq \eta_1, \quad \| \nabla \mu_0\| \leq \eta_2, \quad \| \vp_0\|_{L^\infty(\Omega)}<1,
\end{equation}
then there exists a unique (global) strong solution to \eqref{CHHS}-\eqref{bdini} such that
\begin{align}
\label{GT1}
&\uu \in L^\infty(0,\infty;\H_\sigma\cap \V), \quad p\in L^\infty(0,\infty;H^2(\Omega)),\\
\label{GT2}
&\vp \in L^\infty(0,\infty;W^{2,p}(\Omega))\cap H^1(0,\infty;V),\\
\label{Ginf3d}
&\vp \in L^{\infty}(\Omega \times (0,\infty)) \text{ with } |\vp(x,t)|<1 \text{ a.e. }(x,t) \in \Omega\times(0,\infty),\\
\label{GT3}
&\mu \in L^{\infty}(0,\infty;V)\cap L^2(0,\infty;H^3(\Omega)),
\end{align}
where $2\leq p\leq 6$, and satisfies \eqref{CHHS} almost everywhere in  $\Omega \times (0,\infty)$ and \eqref{bdini} almost everywhere in $\partial \Omega \times (0,\infty)$, and $\vp(\cdot,0)=\vp_0(\cdot)$. Moreover, there exists $\gamma>0$ and $\delta_0>0$ such that 
\begin{equation}
\label{Sepuni}
\| \nabla \mu(t)\|+\| \uu(t)\|\leq C\mathrm{e}^{\frac{\gamma}{2}t},\quad \| \vp(t)\|_{L^\infty(\Omega)}\leq 1-\delta_0, \quad \forall \, t\geq 0,
\end{equation}
for some constant $C>0$ depending on $\eta_1$ and $\eta_2$.
\end{itemize}
\end{theorem}

\begin{remark}
Examples of initial conditions satisfying the smallness conditions in \eqref{assum3} are constant functions sufficiently close to $\beta$ or perturbations of the free energy minimizers. Note that if $\vp_\beta\equiv \beta$, then $\widetilde{\E}(\vp_\beta)=0$ and $\nabla \mu_0=\widetilde{\Psi}'(\vp_\beta)=0$.
\end{remark}

\begin{proof}
We divide the proof into three parts.  
\smallskip

\textbf{First part: Local existence for large data.} 
We follow the argument employed in the proof of Theorem \ref{strong2d}. Notice that Steps $1$-$5$ can be repeated in the same manner in the three dimensional case.
We point out that the solution $(\uu_\varepsilon,p_\varepsilon,\vp_\varepsilon)$ of \eqref{CHHS-reg} is local in time, namely it satisfies \eqref{regus}-\eqref{regmus} on the interval $[0,T^\ast]$, where $T^\ast$ depends on $k$ and $\varepsilon$. This result is proven in \cite[Theorem 3.1]{WZ13}. 

Now we proceed by showing that the solution $(\uu_\varepsilon,p_\varepsilon, \vp_\varepsilon)$ is well-defined on a time interval $[0,T]$, where $T$ is positive and independent of $\varepsilon$ and $k$. In order to do this, we prove uniform higher order estimates. We recall the differential equality
\begin{align}
\ddt  H
 &+  \| \nabla \partial_t\vp_\varepsilon\|^2+
  \int_\Omega F_\varepsilon''(\vp_\varepsilon)|\partial_t \vp_\varepsilon|^2 \, \d x \notag \\
  &=\theta_0 \|\partial_t \vp_\varepsilon\|^2+
  \int_\Omega \mu_\varepsilon\nabla \partial_t \vp_\varepsilon \cdot \uu_\varepsilon \, \d x  - \frac12 \int_{\Omega} \nu'(\vp_\varepsilon) \partial_t \vp_\varepsilon |\uu_\varepsilon|^2 \, \d x, 
 \label{Lambda-eq3D}
\end{align}
for almost every $t\in (0,T^\ast)$, where
$$
H(t)= \frac12 \| \nabla \mu_\varepsilon(t)\|^2 +  \frac12  \int_{\Omega} \nu(\vp_\varepsilon(t))|\uu_\varepsilon(t)|^2\, \d x.
$$
We recall that 
\begin{equation}
\label{Lambda-cont3D}
\frac{1}{C}\big( \| \nabla \mu_\varepsilon\|^2+ \| \uu_\varepsilon\|^2 \big)\leq  H \leq C \big( \| \nabla \mu_\varepsilon\|^2+ \| \uu_\varepsilon\|^2 \big).
\end{equation}
Moreover, we report the bounds
\begin{equation}
\label{E2-3d}
\|\vp_\varepsilon \|_{L^\infty(0,T^\ast;V)}\leq C_0,\quad
\| \nabla \mu_\varepsilon\|_{L^2(0,T^\ast;H)}\leq C_0,\quad
\| \uu_\varepsilon\|_{L^2(0,T^\ast;\H_\sigma)}\leq C_0,
\end{equation}
where the constant $C_0$ only depends on $R_1$,
and the estimates
\begin{equation}
\label{E3-3d}
\| \vp_\varepsilon\|_{H^2(\Omega)}^2\leq C(1+\| \nabla \mu_\varepsilon\|),\quad \| \mu_\varepsilon\|_{V}\leq C(1+\| \nabla \mu_\varepsilon\|),\quad \| \vp_\varepsilon\|_{W^{2,6}(\Omega)}\leq C(1+\| \nabla \mu_\varepsilon\|),
\end{equation}
where the constant $C>0$ is independent of $k$ and $\varepsilon$. 
Note that the latter estimate in \eqref{E3-3d} differs from \eqref{E5} due to the Sobolev embedding $V\hookrightarrow L^6(\Omega)$ in dimension three (cf. Theorem \ref{ell2}).

By using \eqref{Ad3}, \eqref{Lambda-cont3D}, \eqref{E2-3d} and \eqref{E3-3d}, we obtain
\begin{align}
\| \partial_t \vp_\varepsilon\|_{V_0'}
&\leq \| \uu_\varepsilon\|\| \vp_\varepsilon\|_{L^\infty(\Omega)}+
\| \nabla \mu_\varepsilon\|\notag\\
&\leq C\| \uu_\varepsilon\| \| \vp_\varepsilon\|_{H^2(\Omega)}^\frac12 +
\| \nabla \mu_\varepsilon\|\notag \\
&\leq C\| \uu_\varepsilon\| (1+\| \nabla \mu_\varepsilon \|)^\frac14 +
C \| \nabla \mu_\varepsilon\| \notag \\
&\leq  C H^\frac12 + C H^{\frac58}.
\label{vpt3D}
\end{align}
We consider the equation for the vorticity of $\uu_\varepsilon$ that reads in three dimensions as follows
$$
\nu(\vp_\varepsilon)\, \mathrm{curl}\, \uu_\varepsilon +
\nu'(\vp_\varepsilon) \nabla \vp_\varepsilon \times \uu_\varepsilon=
\nabla \mu_\varepsilon \times \nabla \vp_\varepsilon.
$$
By exploiting \eqref{GN3}, \eqref{Lambda-cont3D} and \eqref{E3-3d},
we have
\begin{align*}
\nu_\ast \| \mathrm{curl}\, \uu_\varepsilon \|
&\leq C \big( \| \uu_\varepsilon \|+\| \nabla \mu_\varepsilon\|\big)
\| \nabla \vp_\varepsilon\|_{L^\infty(\Omega)}\\
&\leq C \big( \| \uu_\varepsilon \|+\| \nabla \mu_\varepsilon\|\big)
\| \vp_\varepsilon\|_{W^{2,6}(\Omega)}^{\frac34}\\
&\leq C \big( \| \uu_\varepsilon \|+\| \nabla \mu_\varepsilon\|\big) (1+\| \nabla \mu_\varepsilon\|^\frac34)\\
&\leq CH^\frac12 +C H^{\frac78}.
\end{align*}
Hence, we infer from \eqref{rot} that
\begin{equation}
\label{uV3D}
\| \uu_\varepsilon\|_V\leq CH^\frac12 +C H^\frac78.
\end{equation}
Next, we control the three terms on the right-hand side of \eqref{Lambda-eq3D}. By \eqref{O} and \eqref{vpt3D}, we find
\begin{align}
\label{3d-1}
\theta_0 \| \partial_t \vp_\varepsilon\|^2\leq \frac18 \| \nabla \partial_t \vp_\varepsilon\|^2 +C H +C H^{\frac54}.
\end{align}
By integrating by parts, and using \eqref{L3}, \eqref{I}, the regularity of the Neumann problem and \eqref{CHHS-reg}$_3$, 
the second term on the right-hand side of \eqref{Lambda-eq3D} is controlled as follows
\begin{align}
\int_\Omega \mu_\varepsilon\nabla \partial_t \vp_\varepsilon \cdot \u_\varepsilon \, \d x
&=-\int_\Omega \partial_t \vp_\varepsilon \, \uu_\varepsilon \cdot \nabla \mu_\varepsilon \, \d x\nonumber\\
&\leq \| \u_\varepsilon\| \| \partial_t \vp_\varepsilon\|_{L^3(\Omega)}\| \nabla \mu_\varepsilon\|_{\mathbf{L}^6(\Omega)}\nonumber\\
&\leq C \| \u_\varepsilon\| \|\partial_t \vp_\varepsilon\|^\frac12\| \nabla \partial_t \vp_\varepsilon\|^\frac12 \big(\| \Delta \mu_\varepsilon\| + \|\mu_\varepsilon\| \big) \nonumber\\
&\leq C \| \u_\varepsilon\| \| \partial_t \vp_\varepsilon\|_{V_0'}^\frac14
\| \nabla \partial_t \vp_\varepsilon\|^\frac34 \big( \| \partial_t \vp_\varepsilon\|+ \| \u_\varepsilon\cdot \nabla \vp_\varepsilon\| + \|\mu_\varepsilon \| \big) \notag \\
&= R_1+R_2+R_3.
\label{3d-2}
\end{align}
By \eqref{GN3}, \eqref{Lambda-cont3D}, \eqref{E3-3d} and \eqref{vpt3D}, we have
\begin{align}
R_2&\leq C \| \u_\varepsilon\|^2 \| \partial_t\vp_\varepsilon\|_{V_0'}^{\frac14}
\| \nabla \partial_t \vp_\varepsilon\|^{\frac34}
 \|\nabla \vp_\varepsilon\|_{L^\infty(\Omega)}\notag \\
&\leq \frac{1}{12} \| \nabla \partial_t \vp_\varepsilon\|^2
+ C \| \u_\varepsilon\|^{\frac{16}{5}}
\| \partial_t \vp_\varepsilon\|_{V_0'}^{\frac25}
\|\vp_\varepsilon\|_{W^{2,6}(\Omega)}^\frac65 \notag \\
&\leq \frac{1}{12} \| \nabla \partial_t \vp_\varepsilon\|^2
+ C \| \u_\varepsilon\|^{\frac{16}{5}}
\| \partial_t \vp_\varepsilon\|_{V_0'}^{\frac25}
(1+\| \nabla \mu_\varepsilon\|^\frac65) \notag \\
&\leq \frac{1}{12} \| \nabla \partial_t \vp_\varepsilon \|^2
+C H^{\frac85}(H^\frac15+H^\frac14)(1+H^\frac35)\notag \\
&\leq \frac{1}{12} \| \nabla \partial_t \vp_\varepsilon \|^2
+C\big(1+ H^{\frac{49}{20}}\big).
\label{3d-3}
\end{align}
The remainder terms $R_1$ and $R_3$ can be estimated in the same way. Regarding the third term on the right-hand side of \eqref{Lambda-eq3D}, by exploiting \eqref{L3}, \eqref{I} and \eqref{uV3D}, we find 
\begin{align}
\frac12 \int_{\Omega} \nu'(\vp_\varepsilon) \partial_t \vp_\varepsilon |\uu_\varepsilon|^2 \, \d x& \leq C \| \partial_t \vp_\varepsilon\|
\| \uu_\varepsilon\|_{L^4(\Omega)}^2\notag \\
&\leq C \| \partial_t \vp_\varepsilon\|_{V_0'}^\frac12 
\| \nabla \partial_t \vp_\varepsilon\|^\frac12 \|\uu_\varepsilon\|^\frac12 \|\uu_\varepsilon \|_V^\frac32\notag \\
&\leq \frac18 \| \nabla \partial_t \vp_\varepsilon\|^2+ 
C \| \partial_t \vp_\varepsilon\|_{V_0'}^\frac23 
\|\uu_\varepsilon\|^\frac23 \|\uu_\varepsilon \|_V^2 \notag \\
&\leq \frac18 \| \nabla \partial_t \vp_\varepsilon\|^2+ 
C H^\frac13(H^\frac13+H^{\frac{5}{12}})(H+H^\frac74) \notag \\
&\leq \frac18 \| \nabla \partial_t \vp_\varepsilon\|^2+ 
C \big(1+ H^\frac{5}{2}\big).
\label{3d-4}
\end{align}
Collecting the above estimates \eqref{3d-1}-\eqref{3d-4} together, we deduce the differential inequality
\begin{equation}
\label{Lambda3d-final}
\ddt H +\frac12 \| \nabla \partial_t \vp_\varepsilon\|^2
\leq C_1\big(1+H^\frac{5}{2}\big),
\end{equation}
where $C_1$ is independent of $\varepsilon$ and $k$. 
Besides, in light of \eqref{lambda0}, we have the control on the initial condition
$$
1+H(0)\leq C_2(1+R_1+R_2)^2.
$$
Thus, integrating the above differential inequality, we infer that
$$
H(t)\leq \frac{C_2(1+R_1+R_2)^2}{\big(
1-\frac32C_1C_2^\frac32 (1+R_1+R_2)^3 t \big)^{\frac23}}, \quad \forall \, 
t \in \Big[0,\frac{1}{\frac32 C_1C_2^\frac32 (1+R_1+R_2)^3}\Big).
$$
In particular, this implies
$$
\| \nabla \mu_\varepsilon(t)\|^2+\| \uu_\varepsilon(t)\|^2\leq 2C_2(1+R_1+R_2)^2, \quad \forall \, t\in [0,T], \ \text{ with } \ T_0=\frac{1}{3C_1 C_2^\frac32 (1+R_1+R_2)^3}.
$$
Besides, integrating \eqref{Lambda3d-final} on $[0,T_0]$, we obtain
$$
\int_0^{T_0} \| \nabla \partial_t \vp_\varepsilon(\tau)\|^2 \, \d \tau
\leq C_3(1+R_1+R_2)^2+ C_4 (1+R_1+R_2)^5 T_0. 
$$
for some positive constants $C_3, C_4$ independent of $\varepsilon$ and $k
$. As a consequence, using \eqref{E3-3d}, \eqref{vpt3D} and \eqref{uV3D}, we deduce that
\begin{align*}
\| \uu_\varepsilon(t)\|_{V}+
\| \vp_\varepsilon(t)\|_{W^{2,6}(\Omega)}+\| \partial_t\vp_\varepsilon\|_{V_0'}+
\| \Psi_\varepsilon'(\vp_\varepsilon)(t)\|_{L^p(\Omega)}
+\| p_\varepsilon(t)\|_{H^2(\Omega)}
\leq C_5, \quad \forall \,
t\in [0,T_0],
\end{align*}
where $C_5$ only depends on $p$, $R_1$, $R_2$ and $T$.
Also, by the regularity of the Neumann problem, it follows from \eqref{CHHS-reg}$_3$ that
$$
\int_{0}^{T_0} \| \mu_\varepsilon (\tau)\|_{H^3(\Omega)}^2 \, \d \tau \leq C_6.
$$
These uniform bounds with respect to the approximation parameter $\varepsilon$ and $k$ are sufficient to guarantee the existence of a limit triplet $(\uu,p,\vp)$ on $[0,T_0]$ satisfying the system \eqref{CHHS}-\eqref{bdini} and the regularity properties \eqref{LT1}-\eqref{LT2}. 
\smallskip

\textbf{Second part: Uniqueness of strong solutions.} 
We prove the uniqueness of the strong solutions by controlling the difference of two solutions in the dual space $V_0'$ (cf. Theorem \eqref{weak-strong}).
Let us consider two solutions $(\uu_1,p_1,\vp_1)$ and $(\uu_2,p_2,\vp_2)$ defined on the same interval $[0,T_0]$ corresponding to the same initial condition $\vp_0$. We have $\overline{\vp}_1(t)=\overline{\vp}_2(t)$ for any $t\in [0,T_0]$. The following estimates hold
\begin{equation}
\label{regularity3D}
\| \uu_i\|_{L^\infty(0,T_0;V)}\leq C, \quad \|\vp \|_{L^\infty(0,T_0;W^{2,6}(\Omega))}\leq C.
\end{equation}
Setting $\uu=\uu_1-\uu_2$, $p^\ast= p_1^\ast-p_2^\ast$ and $\vp=\vp_1-\vp_2$ (cf. Remark \ref{equiv}), the triplet $(\uu, p ,\vp)$ satisfies \eqref{diff-problem}-\eqref{diff-mu-u}. By taking $v=\An \vp$ in \eqref{diff-problem} and using \eqref{muvp}, we find
$$
\frac12 \ddt \| \vp\|_{V_0'}^2 + \frac12 \| \nabla \vp\|^2\leq 
C \| \vp\|_{V_0'}^2+ I_1+ I_2,
$$
where $I_1$ and $I_2$ are defined in \eqref{Idef}. For the sake of brevity, we will only address the differences with the two dimensional case.
By \eqref{L3}, \eqref{I} and \eqref{regularity3D}, we have
\begin{align}
I_1&\leq \| \uu\|_{L^6(\Omega)} \| \vp\|_{L^3(\Omega)} \| \vp\|_{V_0'} \notag \\
& \leq \frac14 \| \nabla \vp\|^2+ C\| \vp\|_{V_0'}^2.
\end{align}
We recall that $I_2= \sum_{k=1}^4 Z_k$.
By \eqref{O} and \eqref{regularity3D}, we get
\begin{align*}
Z_1 &\leq \| p^\ast\| \| \nabla \vp_1\|_{L^\infty(\Omega)}
\| \P(\vp_2 \nabla \An \vp)\|\\
& \leq C \| p^\ast \| \| \vp\|_{V_0'}.
\end{align*}
In order to control $p^\ast$, we exploit the equality (cf. \eqref{pL2}) 
\begin{align}
\| p^\ast\|^2&= \Big(\nabla \vp_1 \otimes \nabla \vp+\nabla \vp \otimes \nabla \vp_2, \frac{\nu'(\vp_1)}{\nu(\vp_1)^2} \nabla \vp_1 \otimes \nabla q + \frac{1}{\nu(\vp_1)} \nabla \nabla q \Big) \notag \\
&\quad -\Big(  \frac{(\nu(\vp_1)-\nu(\vp_2))}{\nu(\vp_1)} \uu_2, \nabla q\Big),
\label{pL2-3d}
\end{align}
where $q$ is defined in \eqref{q-def} and satisfies $\| q\|_{H^2(\Omega)}\leq C \| p^\ast\|$ due to \eqref{regularity3D}.
By using once again \eqref{regularity3D}, it is easily seen from \eqref{pL2-3d} that $\| p^\ast\| \leq C \| \nabla \vp\|$.
This, in turn, entails
$$
Z_1\leq \frac{1}{32}\|\nabla \vp \|^2+C \| \vp\|_{V_0'}^2.
$$
We recall the formulas \eqref{Z2-def} and \eqref{Z3-def}. By \eqref{O}, \eqref{Ad3}, \eqref{I}, \eqref{N} and \eqref{regularity3D}, we obtain
\begin{align*}
Z_2+Z_3 &\leq C \| \nabla \vp\| \| \vp_2 \nabla \An \vp\|_{V}\\
&\leq C \| \nabla \vp\| \big( \| \nabla \An \vp\|_{V}+ 
\|  \nabla \An \vp\|_{L^\infty(\Omega)} \big)\\
& \leq C \| \nabla \vp\| \big( \| \vp\|+ 
\| \vp\|^\frac12 \| \nabla \vp\|^\frac12 \big)\\
&\leq C \| \vp\|_{V_0'}^\frac12 \| \nabla \vp\|^\frac32+
C \| \vp\|_{V_0'}^\frac14 \| \nabla \vp\|^\frac74\\
&\leq \frac{1}{16} \|\nabla \vp \|^2+ C \| \vp\|_{V_0'}^2.
\end{align*}
From  \eqref{O}, \eqref{L3}, \eqref{I}, \eqref{ipo-nu} and \eqref{regularity3D}, we get
\begin{align*}
Z_4&\leq C\|\vp\|_{L^3(\Omega)}\|\uu \|_{L^6(\Omega)}
\| \vp\|_{V_0'}\\
&\leq \frac{1}{32} \|\nabla \vp \|^2+ C \|\vp \|_{V_0'}^2.
\end{align*} 
Thus, combining the above controls, we are eventually led to the differential inequality
$$
\ddt \| \vp\|_{V_0'}^2\leq C \| \vp\|_{V_0'}^2,
$$
which entails the uniqueness by the Gronwall lemma.
\smallskip

\textbf{Third part: Global existence for small data}. 
The local existence of a strong solution defined on an interval $[0,T_0]$ satisfying \eqref{LT1}-\eqref{LT3} is guaranteed by the first part of the theorem. 
By virtue of the condition on the initial datum $\| \vp_0\|_{L^\infty(\Omega)}= 1-\delta$ for some $\delta>0$, it is possible to deduce from \eqref{LT2} that $\| \vp(t)-\vp_0\|_{L^\infty(\Omega)}\leq C\sqrt{t}$, for $t\in [0,T_0]$, where $C$ depends on $\eta_1$ and $\eta_2$. Hence, there exists a time $T'$ (depending only on $\eta_1$, $\eta_2$ and $\delta$) such that $0<T'\leq T_0$ and $\| \vp(t)\|_{L^\infty(\Omega)}\leq 1-\frac{\delta}{2}$, for all $t\in [0,T']$. By the assumption (A2), we infer that $\Psi''(\vp)\in L^\infty(\Omega \times (0,T'))$. As a consequence, we deduce that $\partial_t \mu\in L^2(0,T';V')$. Together with the regularity $\mu\in L^2(0,T';H^3(\Omega))$, this entails that $\mu \in \mathcal{C}([0,T'], V)$. In addition, it can be inferred from the boundedness of $\Psi''(\vp)$ that $\vp\in \mathcal{C}([0,T'],H^3(\Omega))\cap L^2(0,T';H^5(\Omega))$ and $\uu\in L^2(0,T';\H^3(\Omega))$. Then, by the continuity in time of $\vp$ and $\mu$, by using the equation \eqref{CHHS}$_1$ we deduce that $\sqrt{\nu(\vp)}\uu\in \mathcal{C}([0,T'], \H)$. 
Setting for $t\in [0,T']$
$$
H(t)= \frac12 \| \nabla \mu(t)\|^2+\frac12 \int_{\Omega} \nu(\vp(t)) |\uu(t)|^2 \, \d x,
$$ 
it immediately follows that $H \in \mathcal{C}([0,T'])$. As in the proof of Theorem \ref{strong2d}, the strong solution satisfies for any $0\leq s\leq t \leq T'$
\begin{align}
H(t) &+\int_s^t \| \nabla \partial_t \vp(\tau)\|^2\, \d \tau 
+\int_s^t F''(\vp(\tau)) |\partial_t \vp(\tau)|^2 \, \d \tau\notag\\
&= H(s)+\theta_0 \int_s^t \| \partial_t \vp(\tau)\|^2 \, \d \tau
+ \int_s^t \int_{\Omega} \mu(\tau) \nabla \partial_t \vp(\tau)\cdot \uu(\tau) \, \d x\d \tau \notag \\
&\quad - \frac12 \int_s^t \int_{\Omega} \nu'(\vp(\tau)) \partial_t \vp(\tau) |\uu(\tau)|^2 \,\d x \d \tau.
\label{Lam3g}
\end{align}

Let us now consider the maximal interval $[0,\widetilde{T})$ of existence of the (unique) strong solution satisfying the above properties. We aim to prove that $\widetilde{T}=\infty$ provided that the initial condition is sufficiently small. To do this, we assume that $T_0<\infty$ and proceed by refining the estimate \eqref{Lambda3d-final}.
First, we report some preliminary estimates we will need in the rest of the proof.
By the energy identity \eqref{BEL} and the assumption on the initial condition, we have
\begin{equation}
\label{EI3d}
\widetilde{\E}(\vp(t))+ \int_0^t \|\nabla \mu(\tau) \|^2+ \| \sqrt{\nu(\vp(\tau))}\uu(\tau)\|^2 \, \d \tau \leq \eta_1.
\end{equation}
Next, by assumption (A1), there exists a positive constant $C$ such that
\begin{equation}
\label{Lambda-beup}
\frac{1}{C} \big( \| \nabla \mu\|^2+ \| \uu\|^2 \big)
\leq H \leq C \big( \| \nabla \mu\|^2+ \| \uu\|^2 \big).
\end{equation}
Using \eqref{Linf3d} and \eqref{Lambda-beup}, it easily follows from the equation \eqref{CHHS}$_3$ that 
\begin{align}
\label{Ref-vpt}
\| \partial_t \vp\|_{V_0'}\leq C H^\frac12.
\end{align} 
By \eqref{E3-3d} and \eqref{uV3D}, we recall that
\begin{equation}
\label{rmu1}
\| \mu\|_V\leq C(1+H^\frac12), \quad \| \vp\|_{W^{2,6}(\Omega)}\leq C(1+ H^\frac12), \quad \| \uu\|_V\leq C H^\frac12 +C H^{\frac78}.
\end{equation}
In addition, an application of Theorem \ref{ell2} with $f=\mu+\theta_0 \vp$ yields
\begin{equation}
\label{H23d}
\| \Delta \vp\|^2 \leq C \| \nabla \vp\| (\| \nabla \mu\|+\| \nabla \vp\|).
\end{equation}
By using \eqref{N} and \eqref{EI3d}, we deduce that
\begin{align}
\|\vp \|_{L^\infty(\Omega)}&\leq \| \vp-\overline{\vp}\|_{L^\infty(\Omega)} + |\overline{\vp}| \notag \\
&\leq C\| \vp-\overline{\vp}\|_{H^2(\Omega)}+ |\overline{\vp}| \notag \\
&\leq C \| \Delta \vp\|_{H^2(\Omega)} +|\overline{\vp}| \notag \\
& \leq \widetilde{C}_0 \sqrt{2 \eta_1} (\| \nabla \mu\|+\sqrt{2\eta_1})+m,
\label{infinito}
\end{align}
where $m=\overline{\vp}_0$,
for some constant $\widetilde{C}_0>0$ independent of $\eta_1$ and $\eta_2$.
Now, we control the terms on the right-hand side of \eqref{Lam3g}.
By using \eqref{I} and \eqref{Ref-vpt}, we obtain
\begin{align}
\theta_0 \int_s^t \| \partial_t \vp(\tau)\|^2 \, \d \tau 
&\leq \frac18 \int_s^t \| \nabla \partial_t \vp(\tau)\|^2\, \d \tau+
C \int_s^t \| \partial_t \vp(\tau)\|_{V_0'}^2 \d \tau \notag \\
&\leq \frac18 \int_s^t \| \nabla \partial_t \vp(\tau)\|^2\, \d \tau+
C \int_s^t H(\tau) \, \d \tau.
\label{H1g}
\end{align}
By arguing as in \eqref{3d-2}, we infer from \eqref{GN3}, \eqref{Ref-vpt}, \eqref{rmu1} that
\begin{align*}
\int_{\Omega} \mu \nabla \partial_t \vp \cdot \uu \, \d x
&\leq C \| \uu\| \| \partial_t \vp\|_{V_0'}^\frac14 \| \nabla \partial_t \vp\|^{\frac34} 
\big( 1+\| \partial_t \vp\|+ \| \uu\cdot \nabla \vp\|+ \| \nabla \mu\| \big)\\
&\leq C H^\frac58 \| \nabla \partial_t \vp\|^\frac34 
(1+ H^\frac12 + \| \uu\| \| \nabla \vp\|_{L^\infty} \big)\\
& \leq C H^\frac58 \| \nabla \partial_t \vp\|^\frac34 
(1+ H^\frac12 + H^\frac12 \| \vp\|_{W^{2,6}(\Omega)}^\frac34\big) \\
&\leq C H^\frac58 \| \nabla \partial_t \vp\|^\frac34 
(1+ H^\frac12 + H^\frac78 \big).
\end{align*}
Here we have used that $\|\nabla \vp\|\leq C$. Hence, by Young's inequality we reach
\begin{equation}
\label{H2g}
\int_s^t \int_{\Omega} \mu \nabla \partial_t \vp \cdot \uu \, \d x
\leq \frac18 \int_s^t \| \nabla \partial_t \vp(\tau)\|^2\, \d \tau+
C \int_s^t H(\tau)+ H^\frac95(\tau)+ H^\frac{12}{5}(\tau) \, \d \tau.
\end{equation}
The third term on the right-hand side can be controlled as in \eqref{3d-4}. By \eqref{Ref-vpt} and \eqref{rmu1}, we obtain
\begin{align*}
\frac12 \int_{\Omega} \nu'(\vp(\tau))\partial_t \vp(\tau) |\uu(\tau)|^2 &\leq C \|\partial_t \vp\|_{V_0'}^\frac12 \| \nabla \partial_t \vp\|^\frac12 \| \uu\|^\frac12 \| \uu\|_V^\frac32\\
&\leq C \| \nabla \partial_t \vp\|^\frac12 H^\frac12 \big( H^\frac34 + H^\frac{21}{16}\big).
\end{align*}
Then, we arrive at 
\begin{equation}
\label{H3g}
-\frac12 \int_s^t \int_{\Omega} \nu'(\vp(\tau))\partial_t \vp(\tau) |\uu(\tau)|^2 \leq \frac{1}{8} \int_s^t \| \nabla \partial_t \vp(\tau)\|^2\, \d \tau+
C \int_s^t H^\frac53(\tau)+H^\frac{29}{12}(\tau) \, \d \tau
\end{equation}
Now we multiply \eqref{CHHS}$_1$ by $\uu$ and \eqref{CHHS}$_3$ by $\mu$, we integrate over $\Omega$ and we add up the two equations. Recalling that $\overline{\partial_t \vp}=0$, it follows from \eqref{poincare} and \eqref{ipo-nu} that
$$
\| \nabla \mu\|^2 + \nu_\ast \| \uu\|^2\leq 
(\mu, \vp_t)\leq \frac12 \| \nabla \mu\|^2 +C \|\nabla \partial_t \vp \|^2.
$$ 
Thus, there exists a (small) constant $\gamma>0$ such that 
\begin{equation}
\label{lowreg}
\gamma \int_s^t H(\tau) \, \d \tau \leq \frac18 \int_s^t \| \nabla \partial_t \vp(\tau)\|^2 \, \d \tau.
\end{equation}
Combining \eqref{Lam3g} with \eqref{H1g}, \eqref{H2g}, \eqref{H3g} and \eqref{lowreg}, we are led to 
\begin{align}
H(t) &+ \gamma \int_s^t H(\tau) \, \d \tau 
+ \frac12 \int_s^t \| \nabla \partial_t \vp(\tau)\|^2\, \d \tau\\
&\leq H(s)+  C \int_s^t H(\tau) \,\d \tau +
C \int_s^t H^\frac{9}{5} (\tau)+ H^{\frac{12}{5}}(\tau)+ 
H^\frac53(\tau)+ H^\frac{29}{12}(\tau) \, \d \tau,
\end{align}
for every $0\leq s<t<\widetilde{T}$.
Let us take $s=0$. We control the first two terms on the right-hand side by a sufficiently small constant depending on the initial datum. We note that
\begin{align*}
H(0)&\leq C \| \nabla \mu_0\|^2 +C \| \uu(0)\|^2 \leq C \| \nabla \mu_0\|^2 +C \| \nabla \mu_0\| ^2\| \vp_0\|_{L^\infty(\Omega)}^2.
\end{align*}
In light of \eqref{assum3}, this implies that
\begin{equation}
H(0)\leq \widetilde{C}_1 \eta_2^2,
\end{equation}
for some constant $\widetilde{C}_1>0$.
On the other hand, by virtue of \eqref{assum3}, \eqref{EI3d} and the definition of $\widetilde{\E}$, we get
$$
C \int_0^t H(\tau) \d \tau \leq \widetilde{C}_2 \eta_1.
$$
Summing up, we have
\begin{align}
H(t) &+ \gamma \int_0^t H(\tau) \, \d \tau 
+ \frac12 \int_0^t \| \nabla \partial_t \vp(\tau)\|^2\, \d \tau \notag\\
&\leq \widetilde{C}_3(\eta_1+\eta_2^2)+
\widetilde{C}_4 \int_0^t H^\frac{9}{5} (\tau)+ H^{\frac{12}{5}}(\tau)+ 
H^\frac53(\tau)+ H^\frac{29}{16}(\tau) \, \d \tau.
\label{lastdiffineq}
\end{align}
Without loss of generality, we can take $0<\gamma<1$, $\widetilde{C}_3=\max\lbrace{\widetilde{C}_1,\widetilde{C}_2\rbrace}>1$ and $\widetilde{C}_4>1$.
We fix
\begin{equation}
\label{defeps}
\epsilon= \min \Big\lbrace{\frac12, \frac{1}{\widetilde{C}_4^\frac32}} \big( \frac{\gamma}{8} \big)^\frac32, \frac{1-m}{\widetilde{C}_0} \frac{\widetilde{C}_3}{\widetilde{C}_3^\frac12 +1} \Big\rbrace,
\end{equation}
where $\widetilde{C}_0$ is the constant in \eqref{infinito}, and we assume the condition  
\begin{equation}
\label{defeta}
\eta_1+\eta_2^2 \leq \frac{\epsilon}{2 \widetilde{C}_3}.
\end{equation}
Let us define 
$$
T_1= \sup \lbrace t \in [0,\widetilde{T}): H(t)\leq \epsilon \rbrace.
$$
Since $H(0)\leq \widetilde{C}_3 \eta_2^2<\epsilon$, the continuity of $H$ guarantees that $T_1>0$.
By the choice of $\epsilon$, we eventually infer that
$$
\widetilde{C}_4 \int_0^t H^\frac{9}{5} (\tau)+ H^{\frac{12}{5}}(\tau)+ 
H^\frac53(\tau)+ H^\frac{29}{12}(\tau) \, \d \tau
\leq \frac{\gamma}{2} \int_0^t H (\tau) \d \tau,
$$
which, in turn, implies 
$$
H(t) +\frac12 \int_0^t \| \nabla \partial_t \vp(\tau)\|^2 \, \d \tau \leq  \frac{\epsilon}{2}, \quad \forall \, t \in [0,T_1].
$$
Here we have used \eqref{defeta}. By the definition of $T_1$ and the continuity of $H$, we reach a contradiction. This implies that 
\begin{equation}
\label{Lambdaeps}
H(t)< \epsilon, \quad \forall t \in [0,\widetilde{T}),
\end{equation}  
and, as a consequence, 
\begin{equation}
\label{Inteps}
\int_0^t \| \nabla \partial_t \vp(\tau)\|^2 \, \d \tau < \epsilon, \quad \forall t \in [0,\widetilde{T}).
\end{equation}
In order extend the solution beyond $\widetilde{T}$, we need to show that the solution is defined in $\widetilde{T}$ and $\| \vp(\widetilde{T})\|_{L^\infty(\Omega)}<1$. We observe that combining \eqref{Lam3g} with \eqref{H1g}, \eqref{H2g}, \eqref{H3g}, we have
\begin{align*}
|H(t)-H(s)|\leq \frac38 \int_s^t \| \nabla \partial_t \vp(\tau)\|^2 \, \d \tau +\widetilde{C}_5 \int_s^t H(\tau) + H^\frac{9}{5} (\tau)+ H^{\frac{12}{5}}(\tau)+ H^\frac53(\tau)+ H^\frac{29}{12}(\tau) \, \d \tau,
\end{align*}
for some constant $\widetilde{C}_5>0$. By \eqref{Lambdaeps}, we obtain
\begin{align*}
|H(t)-H(s)|\leq \frac38 \int_s^t \| \nabla \partial_t \vp(\tau)\|^2 \, \d \tau +\widetilde{C}_6 (t-s),
\end{align*}
where $\widetilde{C}_6$ depends on $\widetilde{C}_5$ and $\epsilon$. Thanks to \eqref{Inteps}, this inequality implies that for any positive $\overline{\varepsilon}>0$, there exists some $\overline{\delta}$ such that 
$$
|H(t)-H(s)|\leq \overline{\varepsilon}, \quad \forall \, t,s \in [0,\widetilde{T}): \ |t-s|< \overline{\delta}.
$$
Hence, $H(t)$ is a Cauchy sequence as $t$ tends to $\widetilde{T}$ and we deduce that $H(\widetilde{T})\leq \epsilon$. 
Besides, we infer from \eqref{infinito}, \eqref{defeps} and \eqref{defeta} that
$$
\| \vp(\widetilde{T})\|_{L^\infty(\Omega)}\leq \widetilde{C}_0 \sqrt{2\eta_1}(\sqrt{\epsilon}+\sqrt{2\eta_1})+m <1.
$$
Thanks to the first and second parts of this proof, we can extend the strong solution on the interval $[\widetilde{T}, \widetilde{T}+ t_0)$ for some $t_0>0$, which is a contradiction to the definition of $\widetilde{T}$. 
Therefore, we conclude that the strong solution is defined globally in time. It is immediate to deduce from \eqref{Lambdaeps} that
$$
H(t)\leq \epsilon, \quad \forall \, t \geq 0.
$$
This control can be further refined by virtue of \eqref{lastdiffineq}. Indeed, by the Gronwall lemma it follows
$$
H(t)\leq \epsilon \mathrm{e}^{-\gamma t}, \quad \forall \, t \geq 0.
$$
Finally, by exploiting \eqref{infinito}, \eqref{defeta} and the above inequality, we find 
$$
\| \vp(t)\|_{L^\infty(\Omega)} \leq \frac{\widetilde{C}_0}{\widetilde{C}_3} \epsilon \big(1+\mathrm{e}^{-\frac{\gamma}{2}t} \big)+ m \leq 1-\delta_0, \quad \forall \, t \geq 0, 
$$
for some $\delta_0>0$.
The proof is complete. 
\end{proof}

\begin{remark}
Notice that the assumption $\|\vp_0 \|_{L^\infty(\Omega)}<1$ in \eqref{assum3} is actually a consequence of the smallness conditions $\widetilde{E}(\vp_0)\leq \eta_1$ and $\| \nabla \mu_0\|\leq \eta_2$. This follows from \eqref{infinito} by replacing $\vp$ with $\vp_0$. 
\end{remark}

\begin{remark}
To the best of our knowledge, the global existence of strong solutions satisfying the separation property \eqref{Sepuni} for all time, provided that the initial condition is sufficiently small, is a novel result even for the Cahn-Hilliard equation with logarithmic potential in dimension three.
\end{remark}

\section{Concluding Remarks}
\label{S7}
\setcounter{equation}{0}

In this work we addressed the well-posedness for Hele-Shaw-Cahn-Hilliard system with unmatched viscosities and physically relevant free energy density of logarithmic type. In dimension two we have proved the existence and uniqueness of global in time strong solutions. In dimension three we have shown the existence and uniqueness of strong solutions, which are local in time for large data or global in time  for small initial data. We have also provided a criterion for the uniqueness of weak solutions in dimension two. Furthermore, we have proved the uniqueness of weak solutions in dimension two when the logarithmic potential is approximated by the classical fourth order polynomial. It is worth noticing that the results here established can be generalized by adding the term $\rho(\vp)\mathbf{g}$ on the right-hand side of the Darcy's law, which takes the difference of densities into account (see \cite[Eqn. (2.26)]{LLG2002} and \cite[Eqn. (2.14)]{DGL}). We observe in conclusion that, even though the HSCH system has been derived as an approximation of the Navier-Stokes-Cahn-Hilliard system, in light of the recent results obtained in \cite{GMT2018} (see also \cite{A}), the mathematical analysis of the HSCH system presents more complex issues. Some interesting problems for the HSCH system remains still unsolved and deserve future investigations, such as uniqueness of weak solutions in dimension two, weak-strong uniqueness and blow-up criteria in dimension three, the analysis of the longtime behavior, the convergence to stationary points, and the formulation of optimal control problems.

\appendix
\section{Generalized Gronwall Lemmas}
\label{App}
\setcounter{equation}{0}

\noindent
We report two Gronwall type lemmas. 

\begin{lemma}
\label{GL2}
Let $f$ be a positive absolutely continuous function on $[0,T]$ and $g$, $h$ two summable functions on $[0,T]$ which satisfy the differential inequality
$$
\ddt f(t)\leq g(t)f(t)\log\big( e+ f(t)\big)+h(t)
$$
for almost every $t\in [0,T]$. Here $C$ is a positive constant. 
Then, we have
$$
f(t)\leq  \big( \mathrm{e}+f(0)\big)^{\mathrm{e}^{\int_0^t g(\tau)\, \d \tau}} \mathrm{e}^
{ \int_0^t \mathrm{e}^{\int_\tau^t g(s)\, \d s} h(\tau) \, \d \tau},
\quad \forall 
\, t \in [0,T].
$$
\end{lemma}
\begin{proof}
We rewrite the differential inequality satisfied by $f$ as follows
$$
\ddt \big( \mathrm{e}+f(t)\big) \leq g(t)\big( \mathrm{e}+ f(t)\big) \log\big( \mathrm{e}+ f(t)\big)+h(t),
$$ 
for almost every $t \in (0,T)$.
Since $f$ is positive, we divide the above inequality by $ (\mathrm{e}+f)$ and we get
\begin{equation}
\label{diffineq}
\ddt \log \big( \mathrm{e}+f(t)\big) \leq g(t) \log\big( \mathrm{e}+  f(t)\big)+h(t),
\end{equation}
for almost every $t\in (0,T)$.
Here we have used a classical result on the composition of a regular function with an absolutely continuous function. Setting $S(t)= \log \big( \mathrm{e}+f(t)\big)$, an application of the Gronwall lemma yields
$$
S(t)\leq S(0)\mathrm{e}^{\int_0^t g(\tau)\, \d \tau}+
 \int_0^t \mathrm{e}^{\int_\tau^t g(s)\, \d s} h(\tau) \, \d \tau, \quad \forall 
\, t \in [0,T].
$$
By definition of $S(t)$, computing the exponential of both 
sides, we deduce the desired conclusion.
\end{proof}

\begin{lemma}
\label{UGL2}
Let $f$ be an absolutely continuous positive function on $[0,\infty)$
and $g$, $h$ two positive locally summable functions on $[0,\infty)$ which satisfy the differential inequality
$$
\ddt f(t) \leq g(t)f(t)\log\Big( \mathrm{e}+ f(t)\Big) +h(t),
$$
for almost every $t\geq 0$, and the uniform bounds
$$
\int_t^{t+r} f(\tau)\, \d \tau\leq a_1, 
\quad \int_t^{t+r} g(\tau)\, \d \tau \leq a_2,
\quad \int_t^{t+r} h(\tau)\, \d \tau \leq a_3, \quad \forall \, t \geq 0,
$$
for some $r$, $a_1$, $a_2$, $a_3$ positive. 
Then, we have
$$
f(t)\leq \mathrm{e}^{\big(\frac{a_1}{r}+a_3\big)\mathrm{e}^{a_2}}, \quad \forall \, t \geq r.
$$
\end{lemma}
\begin{proof}
As in Lemma \ref{GL2}, we rewrite the differential inequality as follows
$$
\ddt \log(\mathrm{e}+f(t))\leq g(t) \log(\mathrm{e}+f(t))+  h(t).
$$
Observing that $\log(\mathrm{e}+ x)\leq x$, for all $x>0$, the uniform Gronwall lemma \cite[Chapter III, Lemma 1.1]{Temam1997} entails
$$
\log(\mathrm{e}+f(t)) \leq \Big( \frac{ a_1}{r}+a_3\Big) \mathrm{e}^{a_2}, \quad \forall \, t \geq r.
$$
The desired conclusion easily follows from the above inequality. 
\end{proof}

\section*{Acknowledgments}
\noindent
The author wish to thank Helmut Abels, Maurizio Grasselli and Giulio Schimperna for some fruitful discussions. This work was supported by GNAMPA-INdAM Project "Analisi matematica di modelli a interfaccia diffusa per fluidi complessi" and by the project Fondazione Cariplo-Regione Lombardia MEGAsTAR ``Matematica d'Eccellenza in biologia ed ingegneria come acceleratore di una nuova strateGia per l'ATtRattivit\`{a} dell'ateneo pavese".

\end{document}